\documentclass[12pt, reqno]{amsart}
\usepackage{graphicx}
\usepackage{hyperref}
\usepackage[caption = false]{subfig}
\usepackage{amsthm}
\usepackage{amssymb}
\usepackage{mathrsfs}
\usepackage{mathtools}
\usepackage{enumerate}
\usepackage{amssymb}
\usepackage{wrapfig}
\usepackage{float}
\allowdisplaybreaks

\usepackage[twoside,paperwidth=190mm, paperheight=297mm, top=35mm, bottom=20mm, left=20mm, right=20mm, marginparsep=3mm, marginparwidth=40mm]{geometry}

\numberwithin{equation}{section}
\DeclareMathOperator{\RE}{Re}
\DeclareMathOperator{\IM}{Im}
\theoremstyle{plain}

\newtheorem{theorem}{Theorem}[section]
\newtheorem{corollary}[theorem]{Corollary}
\newtheorem{example}[theorem]{Example}
\newtheorem{lemma}{Lemma}[section]

\theoremstyle{definition}

\theoremstyle{remark}
\newtheorem{remark}{Remark}[section]

\makeatother

\begin{document}

\title{A Cardioid Domain and Starlike Functions
}
	\thanks{The work of the second author is supported by University Grant Commission, New-Delhi, India  under UGC-Ref. No.:1051/(CSIR-UGC NET JUNE 2017).}	
	
	\author[S. Sivaprasad Kumar]{S. Sivaprasad Kumar}
	\address{Department of Applied Mathematics, Delhi Technological University,
		Delhi--110042, India}
	\email{spkumar@dce.ac.in}

	\author[Kamaljeet]{Kamaljeet Gangania}
	\address{Department of Applied Mathematics, Delhi Technological University,
		Delhi--110042, India}
	\email{gangania.m1991@gmail.com}

\maketitle	
	
\begin{abstract} 
	We introduce and study a class of starlike functions  defined by
\begin{equation*}
\mathscr{S}^*_\wp:=\left\{f\in\mathcal{A}: \frac{zf'(z)}{f(z)}\prec 1+ze^z=:\wp(z)\right\},
\end{equation*}
where  $\wp$  maps the  unit disk onto a cardioid domain. We find the radius of convexity of $\wp(z)$  and establish the inclusion relations between the class $ \mathscr{S}^*_\wp$ and some well-known classes. Further we derive sharp radius constants  and  coefficient related results for the class  $ \mathscr{S}^*_\wp$.\end{abstract}
\vspace{0.5cm}

\maketitle
	
	\section{Introduction}
\label{intro}
Let $\mathcal{H}$ be the class of  analytic functions defined on $\mathbb{D}=\{z\in\mathbb{C}:|z|<1\}$ and $\mathcal{A}_n \subset \mathcal{H} $ such that $f\in \mathcal{A}_n$ has the form $f(z)=z+\sum_{k=n+1}^{\infty}b_{k}z^{k}$. Let $\mathcal{A}:= \mathcal{A}_1$ and $\mathcal{S}$ be the subclass of $\mathcal{A}$ consisting of univalent functions $f$  having the following power series expansion:
\begin{equation}\label{A_n}
f(z)=z+b_{2}z^{2}+b_{3}z^{3}+\cdots.
\end{equation}
Consider a subclass $\mathcal{P}$ of $\mathcal{H}$ consisting of functions with positive real part with the following power series
\begin{equation}\label{caratheodory}
p(z)= 1+\sum_{n=1}^{\infty}p_nz^n
\end{equation}
and again a subclass $\Omega$ of $\mathcal{H}$ consisting of functions $\omega(z)$ satisfying $|\omega(z)|\leq|z|$ and having the form
\begin{equation}
\label{schwarz_fn}
\omega(z)=\sum_{n=1}^{\infty}c_nz^n.
\end{equation}
Recall that if $f, g\in \mathcal{H}$ satisfies the relation $f(z)=g(\omega(z))$, where $\omega(z)\in \Omega$, then we say that $f$ is subordinate to $g$, written $f\prec g$. If $g$ is univalent, then $f\prec g$ if and only if $f(|z|\leq r)\subseteq g(|z|\leq r)$ for all $r$, where $0\leq r<1$. In 1992, using subordination, Ma and Minda \cite{minda94} introduced the classes of starlike and convex functions:
\begin{equation}\label{m-class}
\mathcal{S}^*(\psi):=\biggl\{f\in\mathcal{A}: \frac{zf'(z)}{f(z)}\prec\psi(z)\biggl\} \quad (z\in\mathbb{D})
\end{equation}
and
\begin{equation}\label{m-cclass}
\mathcal{C}(\psi):=\biggl\{f\in\mathcal{A}: 1+\frac{zf''(z)}{f'(z)}\prec\psi(z)\biggl\} \quad (z\in\mathbb{D}),
\end{equation}
where $\psi \in \mathcal{P}$ such that $\psi(\mathbb{D})$ is symmetric about the real axis and starlike with respect to $\psi(0)=1$ with $\psi'(0)>0$. Thus,  symbolically, it unifies many subclasses of $\mathcal{A}$. For example, if $\psi(z):=(1+z)/(1-z), (1+(1+(1-2\alpha))z)/(1-z)$ and $((1+z)/(1-z))^{\gamma}$, then  the class $\mathcal{S}^*(\psi)$ reduces to the class of starlike functions $\mathcal{S}^{*}$, Robertson \cite{robertson1936} class  $\mathcal{S}^{*}(\alpha)$ of starlike function of order $\alpha$ and Stankiewicz \cite{S.Star.beta} class $\mathcal{SS}^*(\gamma)$ of stronly starlike function of order $\gamma$  respectively, where $0\leq\alpha<1$ and $0 <\gamma\leq 1$. Note that $\mathcal{S}^{*}(0)=\mathcal{SS}^{*}(1)=\mathcal{S}^{*}$. Using the coefficients of $f$ given in \eqref{A_n}, Pommerenke \cite{ch-pom1966} and Noonan and Thomas \cite{Noonan} considered the Hankel determinant $H_q(n)$, defined by
\begin{equation}
H_q(n):=\left|\begin{matrix}
b_n     & b_{n+1} & \dots & b_{n+q-1}\\
b_{n+1} & b_{n+2} & \dots & b_{n+q}\\
\vdots  & \vdots  & \ddots & \vdots\\
b_{n+q-1}&b_{n+q} & \dots & b_{n+2(q-1)}
\end{matrix}\right|,
\end{equation}
where $b_1=1$. Finding the upper bound of $|H_3(1)|$, $|H_2(2)|$  and $|H_2(1)|$ for the functions belonging to various subclasses of $\mathcal{A}$ in $\mathcal{S}$  is a usual phenomenon in GFT. Note that the Fekete-Szeg\"{o} functional $b_3-b_2^2$, coincide with $H_2(1)$, which was studied by  Bieberbach in $1916$. In fact, Fekete-Szeg\"o  considered the generalized functional $b_3-\mu b_2^2$, where $\mu$ is real and $f\in \mathcal{S}$. For the class $\mathcal{S}^*$, it is well-known that $|H_2(2)|\leq1$. Recenlty, bound for the second Hankel determinant, $H_2(2)=b_2b_4-b_3^2$, is obtained by Alarif et al. \cite{H22} for the class $\mathcal{S}^{*}(\psi)$. The estimation of third Hankel determinant is more difficult in comparison with second Hankel determinant, especially when sharp bounds are needed, where the third Hankel determinant is given by
\begin{equation}\label{123}
H_3(1)= \left|
\begin{matrix}
b_1 & b_2 & b_3\\
b_2 & b_3 & b_4\\
b_3 & b_4 & b_5
\end{matrix}
\right|= b_3(b_2b_4-b_3^2)-b_4(b_4-b_2b_3)+b_5(b_3-b_2^2).
\end{equation}
The upper bound for $|H_3(1)|$ can be obtained by estimating each term of \eqref{123}, (see \cite{RajaMalik}).
Babalola \cite{Babalola1}  showed that $|H_3(1)|\leq16$ for the class $\mathcal{S}^*$. In $2018$, Lecko \cite{Lecko1/2} obtained the sharp inequality $|H_3(1)|\leq1/9$ for the class $\mathcal{S}^*(1/2)$. For the class $\mathcal{S}^*$, it is proved in \cite{Leko8/9} that $|H_3(1)|\leq8/9$ (not sharp), which improves the earlier known bound $|H_3(1)|\leq1$ established by Zaprawa \cite{Zaprawa}.  For more work in this direction see \cite{shagun,BAMsim,virendraBell}.\\
\indent \quad In Geometric Function Theory, radius problems have a rich history. Note that the extremal function $z/(1-z)^2$ for the class $\mathcal{S}^*$ is not convex in $\mathbb{D}$. However, it is known that for $0<r\leq 2-\sqrt{3}$, $f(|z|\leq r)$ is a convex domain whenever $f\in\mathcal{S}^*$. Grunsky \cite{Grunsky} showed that for the class $\mathcal{S}$, $\tanh{\pi/4}\approx0.6558$ is the radius of starlikeness. Recall that for the subfamilies $G_1$ and $G_2$ of $\mathcal{A}$, we say that $r_0$ is the $G_1$-radius of the class $G_2$, if $r_0\in(0,1)$ is largest number such that $r^{-1}f(rz)\in G_1$, $0<r\leq r_0$ for all $f\in G_2$.  For more work see \cite{ali12,sinefun,mendi,mendi2exp,naveen14,sokol09}.\\
\indent\quad Many subclasses of $\mathcal{S}^*$ were considered in the past, for an appropriate choice of $\psi$ in \eqref{m-class}. For instance, the interesting regions represented by the functions $\sqrt{1+z}, 1+\sin(z),z+\sqrt{1+z^2}, e^z$ and $2/(1+e^{-z})$ were considered in place of $\psi(z)$ by Sok\'{o}\l \; and Stankiewicz \cite{sokol96}, Kumar et al. \cite{sinefun}, Raina et al. \cite{raina}, Mendiratta et al. \cite{mendi2exp} and Goel and Kumar \cite{Goel} respectively. For $-1 \leq B < A \leq 1,\, \mathcal{S}^*[A,B] := \mathcal{S}^*((1+Az)/(1+Bz))$ is the class of Janowski  starlike functions \cite{janow}. Motivated by the classes defined in \cite{sinefun,Goel,mendi,mendi2exp,raina,naveen14,sokol96}, we consider the class of starlike functions related with the contracted cardioid regions represented by the function $\wp_{\alpha}(z)=1+\alpha ze^z$, $0<\alpha\leq1$. More precisely,
\begin{equation*}
\mathscr{S}^*(\wp,\alpha):=\biggl\{ f\in\mathcal{A} : \frac{zf'(z)}{f(z)}\prec 1+\alpha ze^z\biggl\}\quad (z\in\mathbb{D}).
\end{equation*}
Observe that for $0<\alpha<\beta\leq1$, $\wp_{\alpha}(\mathbb{D})\subset \wp_{\beta}(\mathbb{D})$. Let $\wp_{1}(z):=\wp(z)$. We study in particular the following class: 
\begin{equation}
\label{Eqn:1.1}
\mathscr{S}^*_{\wp}:=\biggl\{ f\in\mathcal{A} : \frac{zf'(z)}{f(z)}\prec\wp(z)\biggl\}\quad (z\in\mathbb{D}).
\end{equation}	
A function $f \in \mathscr{S}^*_\wp$ if and only if there exists a function $p \in \mathcal{P}$ and $p \prec \wp$ such that
\begin{equation}
\label{Eqn:2.1}
f(z)= z \exp\left(\int^z_0\frac{p(t)-1}{t}dt \right).
\end{equation}
If we take $p(z)=\wp(z)$, then we obtain from \eqref{Eqn:2.1} the function
\begin{equation}
\label{Eqn:2.2}
f_1(z):= z \exp(e^z-1) = \sum_{n=0}^{\infty} B_n \frac{z^{n+1}}{n!} = z + z^2 + z^3 + \frac{5}{6}z^4 + \frac{5}{8}z^5 + \frac{13}{30}z^6 + \cdots,
\end{equation}
where $B_n$ are the Bell numbers satisfying the recurrence relation given by
\begin{equation}
\label{bell_rec}
B_{n+1} = \sum^{n}_{k=0} \binom{n}{k} B_k.
\end{equation}

We now state below the following common results meant for  $\mathscr{S}^*_\wp$ using results in \cite{minda94}, by omitting the proof.
\begin{theorem}
	Let $f \in \mathscr{S}^*_\wp$ and $f_1$ as defined in (\ref{Eqn:2.2}). Then
	\begin{itemize}\label{GCR}
		\item [$(i)$] Growth theorem: $-f_1(-|z|)\leq|f(z)|\leq f_1(|z|)$.
		\item [$(ii)$] Covering theorem: $\{w : |w|\leq-f_1(-1)\approx{0.5314}\}\subset f(\mathbb{D})$.
		\item [$(iii)$] Rotation theorem: $|\arg f(z)/z|\leq\max_{|z|=r}\arg{(f_1(z)/z)}$.
		\item [$(iv)$] ${f(z)}/{z}\prec f_1(z)/z$ and  $|f'(z)|\leq f'_1(|z|)$.
	\end{itemize}
\end{theorem}
As a consequence of growth theorem, for $|z|=r$, we obtain
\begin{equation*}
\log\biggl|\frac{f(z)}{z}\biggl|\leq\int_{0}^{r}e^tdt\leq\int_{0}^{1}e^tdt=e-1,
\end{equation*}
which implies $|f(z)|\leq e^{e-1}$ and the bound can not be further improved as $z\exp({e^z-1})$ acts as an extremal function.

In the present work, we discuss the geometric properties of the cardioid domain $\wp(\mathbb{D})$ and the inclusion relationship of $\mathscr{S}^*_\wp$ with the classes $\mathcal{SS}^{*}(\gamma)$, $\mathcal{S}^{*}(\alpha)$ and many more. We also obtain various sharp radius results associated with $\mathscr{S}^*_\wp$. Further, we find the coefficient estimates for $f\in \mathscr{S}^*_\wp$ and the sharp bound for the first five coefficients. Conjecture related to the sharp bound of $n$th coefficient is also posed. We also obtain the estimate for the third Hankel determinant for the class $\mathscr{S}^*_\wp$ using the expression of the carath\'{e}odory coefficient $p_4$ in terms of $p_1$, where the technique has not been exploited much so far. The sharp estimates on third Hankel determinant for the classes of two-fold and three-fold symmetric functions associated with $\mathscr{S}^*_\wp$ is also obtained. Further, coefficient related problems are also discussed.

\section{ Properties of cardioid domain}\label{section2}
Since $1+ze^z$ maps $\mathbb{D}$ onto a starlike domain, our first result aims in finding the radius of convexity of the same:
\begin{theorem}
	The radius of convexity of the function $\wp(z)=1+ze^z$
	is the smallest positive root of the equation $r^3-4r^2+4r-1=0,$ which is given by
	$$r_c=(3-\sqrt{5})/2\approx0.381966.$$
\end{theorem}
\begin{proof} Now it is to find the constant $r_c \in (0,1]$ so that
	\begin{equation} \label{cvx1}
	\RE\biggl(1+\frac{z\wp''(z)}{\wp'(z)}\biggl)>0 \quad (|z|<r_c).
	\end{equation}
	Since
	\begin{align*}
	\RE\biggl(1+\frac{z\wp''(z)}{\wp'(z)}\biggl)
	=\frac{r^3\cos \theta+r^2(3+\cos 2\theta)+4r\cos\theta +1}{1+2r\cos\theta +r^2}
	=:g(r,\theta)
	\end{align*}
	and  $g$ is symmetric about the real axis as $g(r,\theta)=g(r,-\theta)$. Thus we only need to consider $\theta\in[0,\pi]$. Further we have
	$1+2\cos\theta r+r^2> 0$ for $r\in(0,1)$ and $\theta \in [0, \pi]).$
	So we may consider the numerator of $g(r,\theta)$ as
	$$g_N(r,\theta):=r^3\cos \theta+r^2(3+\cos 2\theta)+4r\cos\theta +1.$$ Now to arrive at \eqref{cvx1}, we only need to show
	\begin{equation}\label{cvx2}
	g_N(r,\theta)>0 \quad (r\in(0, r_c)).
	\end{equation}
	It is evident that for any fixed $r=r_0$, $g_N(r_0,\theta)$ attains its minimum at $\theta=\pi$ which is given by
	$ g_N(r_0,\pi)= -r_0^3+4r_0^2-4r_0+1.$
	Since $g_N(0,\pi)>0$ and if $r_c$ is the least positive root of  $r^3-4r^2+4r-1=0$ then \eqref{cvx2} follows and hence the result. \qed
\end{proof}

Using elementary calculus, one can easily find the following sharp bounds that are associated with the function $\wp(z)=1+z e^z,$ which are used extensively in obtaining our subsequent results.
\begin{lemma}
	
	\label{func_bnds} \textbf{\emph{Function Bounds:}}
	\begin{description}
		\item[$(i)$] Let $\wp_R(\theta) = 1+e^{\cos\theta}\cos(\theta + \sin\theta),$ then $\wp_R(\theta_0) \leq \RE \wp(z) \leq 1+e$ where  $ \theta_0\approx 1.43396$ is the solution of $3\theta/2 + \sin\theta = \pi$
		and $\wp_R(\theta_0)\approx 0.136038.$
		\item[$(ii)$] Let 	$\wp_I(\theta) = e^{\cos\theta}\sin(\theta + \sin\theta)$,  then $|\IM \wp(z)| \leq \wp_I(\theta_0)$,
		where $\theta_0\approx 0.645913$ is the solution of  $3\theta/2 + \sin\theta = \pi/2$
		and $\wp_I(\theta_0)\approx 2.10743.$
		\item[$(iii)$] $|\arg \, \wp(z)|=|\arctan( \wp_{I}(\theta)/\wp_{R}(\theta))| \leq (0.89782)\pi/2$.
		\item[$(iv)$]  $|\wp(z)| \leq 1+re^r$, whenever $|z|=r<1. $
	\end{description}
	The bounds are the best possible.
\end{lemma}

The following lemma aims at finding the largest (or smallest) disk centered at the sliding point $(a,0)$ inside (or containing) the cardioid domain $\wp(\mathbb{D}).$
\begin{lemma}
	\label{disk_lem}
	Let $\wp(z)=1+ze^z$. Then  we have
	\begin{enumerate}
		\item
		$
		\{w : |w-a|<r_a\} \subset \wp(\mathbb{D}),
		$
		where
		\[  r_a=
		\left\{
		\begin{array}
		{lr}
		(a-1)+{1}/{e}, &  1-{1}/{e}<a\leq1+(e-e^{-1})/{2}; \\
		e-(a-1),   & 1+(e-e^{-1})/{2}\leq a<1+e.
		\end{array}
		\right.
		\]
		\item
		$
		\wp(\mathbb{D}) \subset \{w : |w-a|<R_a\},
		$
		where
		\[ R_a=
		\left\{
		\begin{array}
		{ll}
		1+e-a,     & 1-1/e < a \leq (e+e^{-1})/2 \\
		\sqrt{d(\theta_a)}, & (e+e^{-1})/2 < a <1+e.
		\end{array}
		\right.
		\] where $\theta_a \in (0,\pi),$ is the root of the following equation:
		\begin{equation}\label{theta-a} \sin(\theta/2)+(1-a)\sin(3\theta/2+\sin\theta)=0.\end{equation}
	\end{enumerate}
\end{lemma}

\begin{proof} We begin with the first part.
	The curve $\wp(e^{i\theta})=1+e^{\cos\theta}(\cos(\theta+\sin\theta)+i\sin(\theta+\sin\theta))$ represents  the boundary of $\wp(\mathbb{D})$ and is symmetric about the real axis. So it is enough to consider $\theta$ in $[0,\pi]$. Now, square of the distance of $(a, 0)$ from the points on the curve $\wp(e^{i\theta})$ is given by
	\begin{align*}
	d(\theta) &:= (a-1-e^{\cos\theta}\cos(\theta+\sin\theta))^2 + e^{2\cos\theta}\sin^2(\theta+\sin\theta)
	\\ &= e^{2\cos\theta} - 2(a-1)e^{\cos\theta}\cos(\theta+\sin\theta) + (a-1)^2.
	\end{align*}
	\textbf{Case(i):} If $1-e^{-1} < a \leq (e+e^{-1})/2$, then $d(\theta)$ decreases in $ [0,\pi]$. Therefore, we get
	$$
	r_a = \displaystyle{\min_{\theta\in[0,\pi]}} \sqrt{d(\theta)} = \sqrt{d(\pi)} = (a-1)+1/e.
	$$
	Now for the range $(e+e^{-1})/2 \leq a < 1+(e-e^{-1})/2$, it is easy to see that the equation
	$$ d'(\theta)= -4e^{\cos\theta}\cos(\theta/2)(\sin(\theta/2)-(a-1)\sin(3\theta/2+\sin\theta))=0$$
	has  three real roots $0,\theta_a$ and $\pi$, where $\theta_a$ is the root of the equation given in (\ref{theta-a})
	and we have $\theta_{a_1} < \theta_{a_2}$ whenever $a_1 < a_2$. Further, we see that $d(\theta)$ increases in $ [0,\theta_a]$ and decreases in $  [\theta_a,\pi]$. Also,
	$$
	d(\pi)-d(0)=2(e+e^{-1})(a-(1+(e-e^{-1})/2))<0.
	$$
	Therefore, $\min\{d(0),d(\theta_a),d(\pi)\}=d(\pi)$ and we have $$r_a=\sqrt{d(\pi)}=(a-1)+1/e.$$
	\textbf{Case(ii):} If $1+(e-e^{-1})/2\leq a<1+e$, we see that $d(\theta)$ is an increasing function for $\theta\in[0,\theta_a]$ and decreasing for $\theta\in[\theta_a,\pi]$, where $\theta_a$ is the root of the equation defined in (\ref{theta-a}). Also,\\
	$$
	d(\pi)-d(0)=2(e+e^{-1})(a-(1+(e-e^{-1})/2))>0.
	$$
	Therefore, $\min\{d(0),d(\theta_a),d(\pi)\}=d(0)$ and we have $$r_a=\sqrt{d(0)}=e-(a-1).$$
	This completes the proof of first part.
	The proof of second part is much akin to the first part so is skipped here.\qed
\end{proof}

\begin{remark}
	We obtain the largest disk $D_L:=|w-a|<r_a$ contained in $\wp(\mathbb{D})$  when $a=1+(e-e^{-1})/2$ and $r_a=(e+e^{-1})/2$ and the smallest disk $D_S := |w-a|<R_a,$ which contains $\wp(\mathbb{D})$  when $a=(e+e^{-1})/2$ and $R_a = 1+(e-e^{-1})/2$. Thus $D_L \subset \wp(\mathbb{D}) \subset D_S$.
\end{remark}

\noindent Our next result deals with the inclusion relations of the class $\mathscr{S}^*_\wp$ involving various classes including the following (see \cite{ali08,kanas,uralegaddi}):
\begin{equation*}
\mathcal{M}(\beta) := \biggl\{ f \in \mathcal{A}: \RE \frac{zf'(z)}{f(z)} < \beta,\quad  \beta>1 \biggl\},
\end{equation*}

\begin{equation*}
k-\mathcal{ST} := \biggl\{ f \in \mathcal{A}: \RE \frac{zf'(z)}{f(z)} > k\left| \frac{zf'(z)}{f(z)}-1 \right|,\quad  k \geq 0 \biggl\}
\end{equation*}
and

\begin{equation*}
\mathcal{ST}_p(a) := \biggl\{f \in \mathcal{A}: \RE \frac{zf'(z)}{f(z)} + a > \left|\frac{zf'(z)}{f(z)}-a\right|,\quad a>0 \biggl\}.
\end{equation*}

\begin{theorem}
	\label{incl_rel}
	\textbf{\emph{Inclusion Relations:}}
	\begin{itemize}
		
		\item[$(i)$] $\mathscr{S}^*_{\wp} \subset \mathcal{S}^{*} (\alpha) \subset \mathcal{S}^{*}$ for $0\leq \alpha \leq \omega_0 \approx 0.136038$.
		
		\item[$(ii)$]  $\mathscr{S}^*_\wp \subset \mathcal{M}(\beta)$ \; \text{for}\; $\beta \geq 1+e$ \;\text{and}\; $\mathscr{S}^*_\wp \subset  \mathcal{S}^{*} (\omega_0) \cap M(1+e)$.
		
		\item[$(iii)$] $\mathscr{S}^*_\wp \subset \mathcal{SS}^*(\gamma) \subset \mathcal{S}^*$ for $0.897828\approx \gamma_0 \leq \gamma \leq 1$.
		
		\item[$(iv)$] $\mathscr{S}^*_\wp\subset \mathcal{ST}_p(a)$ for $a \geq b\approx 1.58405$.
		
		\item[$(v)$] $k-\mathcal{ST} \subset \mathscr{S}^*_\wp$ for $k \geq e-1$,
		
	\end{itemize}
	where  $\omega_0 =\min \RE{\wp(z)}$. These best possible inclusion relations are clearly depicted in Figure \ref{f1}.

\end{theorem}

\begin{proof} Proof of  $(i)$, $(ii)$ and $(iii)$ directly follows from Lemma~\ref{func_bnds}. We begin with the proof of part $(iv)$.
	
	$(iv)$ Note that boundary $\partial\Omega_a$ of the domain $\Omega_a=\{w\in\mathbb{C} : \RE w+a >|w-a| \}$ is a parabola. Now $\mathscr{S}^*_\wp \subset\mathcal{ST}_p(a)$, provided $\RE w +a >|w-a|$, where $w=1+ze^z$. Upon taking $z=e^{i\theta}$, we have
	
	$$T(\theta):= \frac{e^{2\cos \theta} \sin^2(\theta+\sin\theta)}{4(1+e^{\cos\theta} \cos(\theta+\sin\theta))} <a.$$
	Further, $T'(\theta)=0$ if and only if $\theta\in \{ 0,\theta_0, \pi\}$, where $\theta_0\approx 1.23442$ is the unique root of the equation
	$$\cos({3\theta}/{2}+\sin\theta)(2+e^{\cos\theta}\cos(\theta+\sin\theta))+e^{\cos\theta}\cos({\theta}/{2})=0,(0<\theta<\pi).$$
	Therefore, $\max_{0\leq\theta\leq\pi}T(\theta)=\max\{T(0), T(\theta_0),T(\pi)\}=T(\theta_0)\approx 1.58405.$ Since $\mathcal{ST}_p(a_1) \subset \mathcal{ST}_p(a_2)$ for $a_1<a_2$, it follows that $\mathscr{S}^*_\wp \subset \mathcal{ST}_p(a)$ for $a\geq b\approx 1.58405$.
	
	$(v)$ Let $f \in k-\mathcal{ST}$ and $\Gamma_k = \{w \in \mathbb{C}: \RE w > k|w-1|\}$.
	For $k>1$, the boundary curve $\partial\Gamma_k$ is an ellipse $\gamma_k: x^2 = k^2(x-1)^2+k^2y^2$ which can be rewritten as
	$$
	\frac{(x-x_0)^2}{u^2} + \frac{(y-y_0)^2}{v^2} = 1,
	$$
	where $x_0=k^2/(k^2-1),\, y_0=0,\, u=k/(k^2-1)$ and $v=1/\sqrt{k^2-1}$. Observe that $u>v$. Therefore, for the ellipse $\gamma_k$ to lie inside $\wp(\mathbb{\overline{D}})$, we must ensure that $x_0\in(1-1/e,1+e),$ which holds for $k\geq \sqrt{(1+e)/e}$ and by Lemma \ref{disk_lem}, we have
	$$1-{1}/{e}\leq x_0-u \quad\text{and}\quad x_0+u\leq 1+e,$$
	whenever
	$$k\geq\max\left\{\sqrt{{(1+e)}/{e}},\; e-1,\; {(1+e)}/{e}\right\}=e-1.$$
	Since $\Gamma_{k_1} \subseteq \Gamma_{k_2}$ for $k_1 \geq k_2$, it follows that $k-\mathcal{ST} \subset \mathscr{S}^*_\wp$ for $k \geq e-1$.\qed
\end{proof}

\begin{remark}\label{v.ellipse}
	{\bf{Inclusion relation of cardioid with vertical Ellipse}:} \\
	Consider the equation $\frac{(x-h)^2}{v^2}+\frac{y^2}{u^2}=1,$ where
	$x=h+v\cos\theta$, $y=u\sin\theta$ and $\theta\in(0,\pi)$. Now { Case(i)} Let $u=\max\IM\wp(z)$ and $v=h-(1-1/e)$, where $h=\RE\wp(z),$ which corresponds to $\max\IM\wp(z)\approx1.70529$ for the largest vertical ellipse, $V_L$ inside $\wp(\mathbb{D})$.
	{ Case(ii)} Let $h=v=(1+e)/2$ and $u=e-(1/2e)$ for the smallest vertical ellipse, $V_S$ containing $\wp(\mathbb{D})$.
\end{remark}

In view of the Remark \ref{v.ellipse} and the class of starlike functions related with the conic domains considered by Kanas and Wi\'{s}niowska \cite{kanas}, we pose the following problem as an independent interest:\\
{\bf{Open Problem.}} \textit{Find the explicit form of the functions $\Psi\in \mathcal{P},$ which maps $\mathbb{D}$ onto a vertical elliptical domain}.

\begin{figure}[h]
	\begin{tabular}{c}
		\includegraphics[scale=0.4]{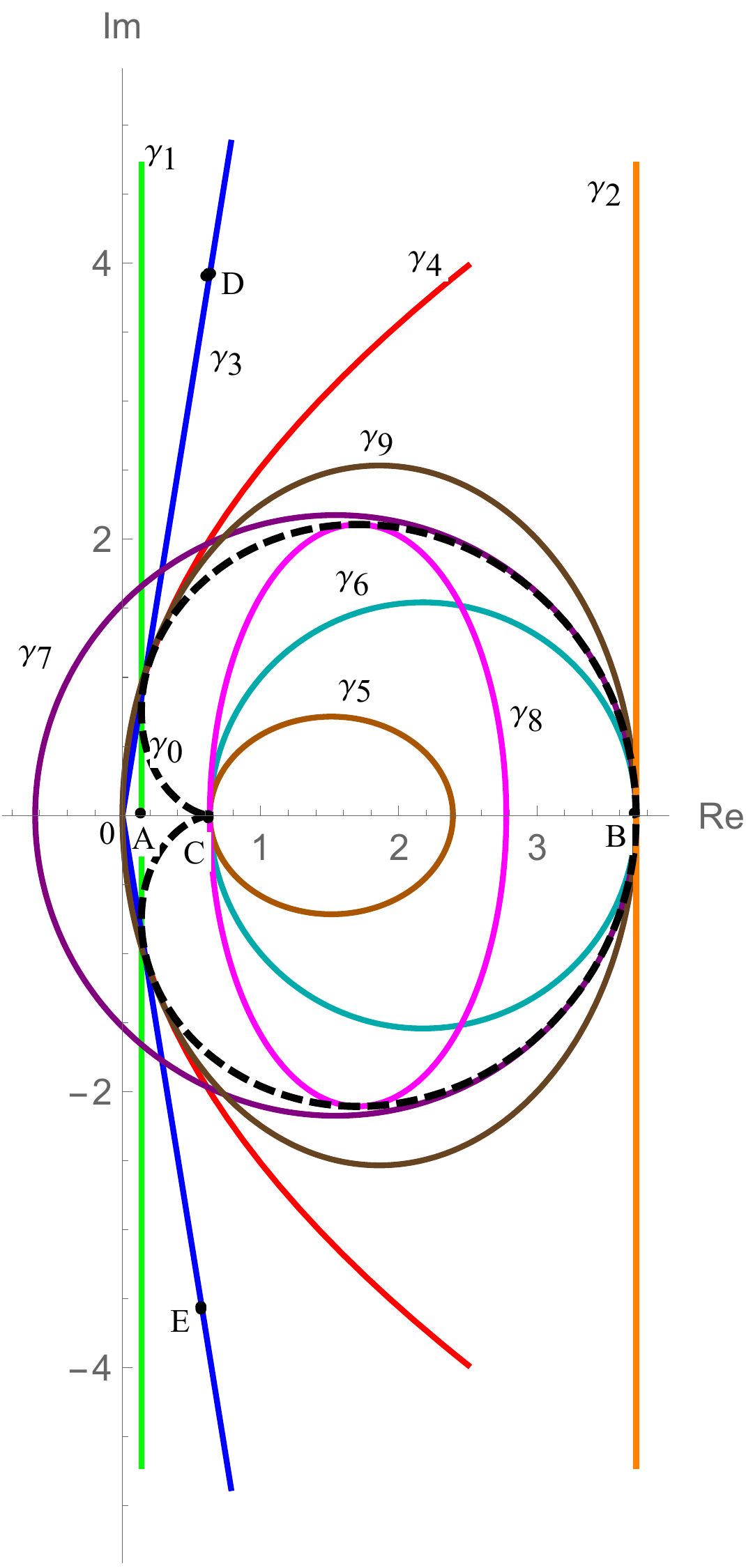}
	\end{tabular}
	\begin{tabular}{l}
		\parbox{0.5\linewidth}{
			{\bf \underline{Legend} -}\\~\\
			$\gamma_0:  \wp(z)=1+ze^z$\\
			$\gamma_1: \RE w=A$\\
			$\gamma_2: \RE w=1+e$\\
			$\gamma_3: |\arg w | = (0.897828)\pi/2$\\
			$\gamma_4: |w-1.58405|-\RE w= 1.58405$\\
			$\gamma_5: \RE w =(e-1)|w-1|$\\
			$\gamma_6: |w-(1+\frac{e^2-1}{2e})| =\frac{(e^2+1)}{2e} $\\
			$\gamma_7: |w-\frac{e^2+1}{2e}| =1+\frac{e^2-1}{2e}$\\
			$\gamma_8:\frac{(\RE w-1.7052)^2}{(2.1074)^2}
			+ \frac{(\IM w)^2}{(1.0731)^2} = 1 $\\
			$\gamma_9: \frac{(2\RE w-B)^2}{B^2}
			+\frac{4(\IM w)^2}{(e+BC)^2} = 1$\\
			$A=\min\RE\wp(z) \approx 0.136038$\\
			$B = 1+e$\\
			$C = 1-e^{-1}$\\
			$\arg(D)=-\arg(E)=(0.897828)\pi/2$
		}
	\end{tabular}
	\caption{Boundary curves of best dominants and subordinants of $\wp(z)=1+ze^z.$}\label{f1}
\end{figure}


\noindent For our next result, we need the following class and some related results:
\begin{equation*}
\mathcal{P}_n[A,B]:= \biggl\{ p(z) = 1 + \displaystyle{\sum_{k=n}^{\infty}}c_nz^n : p(z) \prec \frac{1+Az}{1+Bz},\quad |B|\leq1, A\neq B \biggl\},
\end{equation*}
where $\mathcal{P}_n(\alpha) := \mathcal{P}_n[1-2\alpha,-1]$ and $\mathcal{P}_n:= \mathcal{P}_n(0)\quad (0\leq\alpha<1)$.

\begin{lemma}
	\label{p-nAlpha_lem}\emph{\cite{shah}}
	If $p \in \mathcal{P}_n(\alpha)$, then for $|z|=r$,
	\[
	\left|\frac{zp'(z)}{p(z)}\right| \leq \frac{2(1-\alpha)nr^n}{(1-r^n)(1+(1-2\alpha)r^n)}.
	\]
\end{lemma}

\begin{lemma}
	\label{p-nAB_lem}\emph{\cite{ravi-ron}}
	If $p(z)\in\mathcal{P}_n[A,B]$, then for $|z|=r$,
	\[
	\left|p(z)-\frac{1-ABr^{2n}}{1-B^2r^{2n}}\right| \leq \frac{|A-B|r^n}{1-B^2r^{2n}}.
	\]
	Particularly, if $p \in \mathcal{P}_n(\alpha)$, then
	\[
	\left| p(z) - \frac{1+(1-2\alpha)r^{2n}}{1-r^{2n}}\right| \leq \frac{2(1-\alpha)r^n}{1-r^{2n}}.
	\]
\end{lemma}

\begin{theorem}
	Let $-1 < B < A \leq 1$. If one of the following two conditions hold.
	\begin{itemize}
		\item [$(i)$] $2(e-1)(1-B^2)<2e(1-AB)\leq(e^2+2e-1)(1-B^2)$ and $B-1\leq e(1-A)$;
		\item [$(ii)$] $(e^2+2e-1)(1-B^2)\leq2e(1-AB)<2e(1+e)(1-B^2)$ and $A-B\leq e(1+B)$.
	\end{itemize}
	Then $\mathcal{S}^{*}[A,B]\subset \mathscr{S}^*_\wp$.
\end{theorem}

\begin{proof}
	If $f\in\mathcal{S}^{*}[A,B]$, then $zf'(z)/f(z)\in\mathcal{P}[A,B]$. Therefore, by Lemma \ref{p-nAB_lem}, we have
	\begin{equation}
	\label{p-thm-eq}
	\bigg|\frac{zf'(z)}{f(z)}-a\bigg| \leq \frac{(A-B)}{1-B^2},
	\end{equation}
	where $a:=(1-AB)/(1-B^2)$. Suppose that the conditions in $(i)$ hold. Now multiplying by $(B+1)$ on both sides of the inequality $(B-1) \leq e(1-A)$  and then dividing by $1-B^2$, gives $(A-B)/(1-B^2)\leq a-(1-1/e)$. Similarly, the inequality $2(e-1)(1-B^2)<2e(1-AB)\leq(e^2+2e-1)(1-B^2)$ is equivalent to $1-1/e<(1-AB)/(1-B^2)\leq 1+(e-e^{-1})/2$. Therefore from \eqref{p-thm-eq}, we see that $zf'(z)/f(z) \in \{w\in \mathbb{C}: |w-a|<r_a\} $, where $r_a=a-(1-1/e)$ and $1-1/e<a\leq 1+(e-e^{-1})/2$. Hence, $f\in\mathscr{S}^*_\wp$ by Lemma \ref{disk_lem}. Similarly, we can show that $f\in\mathscr{S}^*_\wp$, if the conditions in $(ii)$ hold. \qed
\end{proof}

\section{\bf Radius Problems}\label{section3}
In this section, we consider several radius problem for $\mathscr{S}^*_\wp$. In the following theorem, we find the largest radius $r_\epsilon<1$, for which, the functions in $\mathscr{S}^*_\wp$ are in a desired class  when $\epsilon$ is given. We denote the $A$-radius of the class $B$ by $R_A(B)$.
\begin{theorem}
	\label{main_res}
	Let $f \in \mathscr{S}^*_\wp$. Then
	\begin{itemize}
		\item [$(i)$] $f\in \mathcal{S}^{*}(\alpha)$ in $|z|<r_\alpha$, $\alpha \in (\alpha_0,1)$,
		where $\alpha_0=1+\tfrac{\sqrt{5}-3}{2}e^{\tfrac{\sqrt{5}-3}{2}}$ and $r_\alpha \in (0,1)$ is the smallest root of the equation
		$$1-re^{-r}-\alpha=0.$$
		
		\item [$(ii)$] $f \in \mathcal{M}(\beta)$ in $|z|<r_\beta$, where
		\[
		r_\beta= \left\{
		\begin{array}{lll}
		r_0(\beta) & $for$ & 1 < \beta < 1+e\\
		1 & $for$ & \beta \geq 1+e	
		\end{array}	
		\right.
		\]
		and  $r_0(\beta)\in (0,1)$ is the smallest root of $1+re^r = \beta$.

		\item [$(iii)$] $f\in \mathcal{SS}^*(\gamma)$ in $|z|<r_\gamma$, $\gamma \in (0,1]$,
		where
		$$r_\gamma=\min\{1,r_0(\gamma)\}$$
		and $r_0(\gamma)\in (0,1)$  is the smallest root of the following equation:
		\begin{equation}
		\label{strstar-eqn}
		\arcsin\left(\frac{1}{r}\ln\left(\frac{r}{\sin(\gamma\pi/2)}\right)\right) + \sqrt{r^2 + \ln^2\left(\frac{r}{\sin(\gamma\pi/2)}\right)} = \frac{\gamma\pi}{2}.		
		\end{equation}
		
	\end{itemize}
\end{theorem}

\begin{proof}
	Since $zf'(z)/f(z) \prec \wp$, it suffices to consider  the cardioid domain $\wp(\mathbb{D})$ so that certain geometry can be performed.
	\begin{enumerate}[(i)]
		
		\item 
		Since $f \in \mathscr{S}^*_\wp$,  there exists a function $\omega(z)\in \Omega$ such that
		$$\frac{zf'(z)}{f(z)}=1+\omega(z)e^{\omega(z)}.$$
		Since $|\omega(z)|\leq|z|$, we can assume $\omega(z)=Re^{i\theta}$, where $R\leq|z|=r$ and $-\pi\leq \theta\leq\pi$. A calculation shows that
		$$|Re^{i\theta}e^{Re^{i\theta}}| = Re^{R\cos{\theta}}=:T(\theta).$$
		Since $T(\theta)=T(-\theta)$, it is sufficient to consider $\theta\in [0,\pi]$. Further, $T'(\theta)\leq0$ implies
		$$|\omega(z) e^{\omega(z)}| =T(\theta) \leq T(0)\leq Re^R\leq re^r.$$
		Therefore, we obtain
		$$ \RE \frac{zf'(z)}{f(z)} \geq 1-|\omega(z)e^{\omega(z)}| \geq 1-re^r \geq \alpha,$$
		whenever $1-re^r-\alpha \geq0$. Hence the result.	
		
		\item 
		Since $f \in \mathscr{S}^*_\wp$. Therefore, using subordination principle and Lemma \ref{func_bnds}, we have
		\begin{equation}
		\label{Eqn:2.5}
		\RE \frac{zf'(z)}{f(z)} \leq\RE\wp(\omega(z))\leq|\wp(\omega(z))|\leq 1+re^r \quad( |z|=r),
		\end{equation}
		where $\omega\in \Omega$. Thus $f \in \mathcal{M}(\beta)$ in $|z|<r$, whenever $1+re^r<\beta$.
		
		\item 
		Let $f \in \mathscr{S}^*_\wp$, then $f\in\mathcal{SS}^{*}(\gamma)$ in $|z|<r$ provided
		$$\left| \arg \frac{zf'(z)}{f(z)} \right| \leq \left| \arg(\wp(z)) \right| \leq \gamma\pi/2 \quad ( |z|= r ).$$
		Assuming $z = r e^{i(\theta + \pi/2)}$,
		\begin{equation}\label{seq1}
		\theta + r \cos\theta = \gamma\pi/2 \;\text{ and }\; r e^{-r \sin\theta} = \sin (\gamma\pi/2),
		\end{equation}
		we have
		\begin{align*}
		1+z e^{z} &=  1 + r e^{-r \sin\theta} (-\sin(\theta + r \cos\theta) + i \cos(\theta + r \cos\theta)) \\
		&=  1 + \sin (\gamma\pi/2) (-\sin (\gamma\pi/2) + i \cos (\gamma\pi/2))\\
		&=  \cos^2 (\gamma\pi/2) + i \sin (\gamma\pi/2) \cos (\gamma\pi/2),
		\end{align*} which implies $\left| \arg(\wp(z)) \right| \leq \gamma\pi/2.$
		Now we obtain  equation~(\ref{strstar-eqn}) by eliminating $\theta$ from the equations, given in (\ref{seq1}), a geometrical observation, ensures  existence of the unique root for the equations~(\ref{strstar-eqn}). Thus the result now follows by considering the smallest root of (\ref{strstar-eqn}). The result is further sharp  as we can find $z_0 = r_0 e^{i(\theta_0 + \pi/2)}$, for any fixed $\gamma$,  at which, the function $f_1,$ given by \eqref{Eqn:2.2},  satisfies
		\begin{align*}
		\left|\arg \frac{z_0f_1'(z_0)}{f_1(z_0)}\right|&=|\arg(1+z e^{z})|_{z=z_0} \\
		&= |\arg(\cos^2 (\gamma\pi/2) + i \sin (\gamma\pi/2) \cos (\gamma\pi/2))|\\
		&=|\arctan(\tan(\gamma\pi/2))|\\
		&=\gamma\pi/2.
		\end{align*} Hence the result.\qed
	\end{enumerate}	
\end{proof}

\begin{theorem}\label{alphaconvex}
	Let $f\in\mathscr{S}^*_\wp$. Then $f\in \mathcal{C}(\alpha)$ in $|z|<r_{\alpha}$,
	where $r_{\alpha}\in (0,1)$ is the smallest root of the equation
	\begin{equation}\label{cr0}
	(1-r)(1-re^r)(1-re^r-\alpha)-re^r=0.
	\end{equation}
\end{theorem}
\begin{proof}
	If $f \in \mathscr{S}^*_\wp$, then there exists a function $\omega\in \Omega$ such that
	$$\frac{zf'(z)}{f(z)}=1+\omega(z)e^{\omega(z)}.$$
	Now a computation yields
	\begin{equation}\label{cr1}
	1+\frac{zf''(z)}{f'(z)}=1+ \omega(z)e^{\omega(z)}+\frac{z \omega'(z) e^{\omega(z)} (1+\omega(z))}{1+\omega(z) e^{\omega(z)}}.
	\end{equation}
	From \eqref{cr1}, we obtain
	\begin{equation}\label{cr2}
	\RE \left(1+\frac{zf''(z)}{f'(z)} \right) \geq 1+ \RE(\omega(z)e^{\omega(z)})-\frac{|z\omega'(z)||(\omega(z)+1)e^{\omega(z)}|}{1-|\omega(z) e^{\omega(z)}|}.
	\end{equation}
	Since $|\omega(z)|\leq|z|$, we can assume that $\omega(z)=Re^{i\theta}$, where $R\leq |z|=r$ and $-\pi \leq \theta \leq \pi$. Now using triangle inequality together with the Schwarz-Pick inequality: $$\frac{|\omega'(z)|}{1-|\omega(z)|^2}\leq \frac{1}{1-|z|^2},$$
	we have
	$|z \omega'(z) e^{\omega(z)} (1+\omega(z))| \leq {r e^r}/{(1-r)}$. Also $|\omega(z)e^{\omega(z)}|\leq Re^R\leq re^r$. Upon using these inequalities in \eqref{cr2}, we get
	\begin{equation} \label{inq6}
	\RE \left(1+\frac{zf''(z)}{f'(z)} \right) \geq 1-re^r -\frac{re^r}{(1-r)(1-re^r)}\geq \alpha,
	\end{equation}
	and with the least root of \eqref{cr0}, the above inequality \eqref{inq6} hold and hence the result.
\end{proof}	
By taking $\alpha=0$ in Theorem~\ref{alphaconvex}, we obtain the following result:
\begin{corollary}
	Let $f\in\mathscr{S}^*_\wp$. Then $f\in \mathcal{C}$ whenever $|z|< r_0\approx 0.256707$ .
\end{corollary}	

\begin{remark}
	Let $\omega(z)=z=re^{i\theta}$ and $\alpha=0$ in Theorem \ref{alphaconvex}. Then for the function given by \eqref{Eqn:2.2}, we have
	$$\RE\left( 1+\frac{zf_{1}''(z)}{f_{1}'(z)} \right) = \RE\left( 1+ze^z + \frac{z(1+z)e^z}{1+ze^z} \right)=:  F(r,\theta),$$
	where
	\begin{align*}
	F(r,\theta)=-\frac{1+r\cos\theta+R\cos\theta_1+rR\cos(\theta_1-\theta)}{1+2R\cos\theta_1+R^2}
	+2 +r\cos\theta + R\cos\theta_1
	\end{align*}
	and $$z=re^{i\theta}, R=re^{r\cos\theta}, \theta_1=\theta+r\sin\theta.$$ Numerically, we note that for all $\theta\in[0,\pi]$, $F(r,\theta) \geq 0$ whenever $r\leq r_0\approx 0.599547$, but when $\theta$ approaches to $\pi$, then $F(r,\theta) < 0$ for $r=r_0+\epsilon,\,\epsilon>0$. Thus we previse  that the sharp radius of convexity for the class $\mathscr{S}^*_\wp$ is $r_0$ .
\end{remark}

For the next theorems \ref{Scn-rad-thm-beg}-\ref{Scn-rad-thm-end}, we need to recall some classes:
Let $f \in \mathcal{A}_n$, if we set $p(z)=zf'(z)/f(z)$, then the class $\mathcal{P}_n[A,B]$ reduces to $\mathcal{S}^*_n[A,B]$, the class of Janowski starlike functions and $\mathcal{S}^*_n(\alpha) := \mathcal{S}^*_n[1-2\alpha,-1]$. Further, let
\begin{equation*}
\mathscr{S}^*_{\wp,n} := \mathcal{A}_n \cap \mathscr{S}^*_\wp \quad\text{and}\quad \mathcal{S}^*_n(\alpha):= \mathcal{A}_n \cap \mathcal{S}^*(\alpha).
\end{equation*}
Ali et al. \cite{ali12} studied the classes, $\mathcal{F}_n := \{f \in \mathcal{A}_n : f(z)/z \in \mathcal{P}_n \},\; \mathcal{S}^*_n[A,B]$ and the subclass consisting of close-to-starlike functions of type $\alpha$ given by

\begin{equation*}
\mathcal{CS}_n(\alpha) :=\biggl \{f \in \mathcal{A}_n : \frac{f(z)}{g(z)} \in \mathcal{P}_n, \; g \in \mathcal{S}^*_n(\alpha) \biggl \}.
\end{equation*}

We find the $\mathscr{S}^*_{\wp,n}$-radii for the classes defined above.

\begin{theorem}
	\label{Scn-rad-thm-beg}
	The sharp $\mathscr{S}^*_{\wp,n}$-radius of the class $\mathcal{F}_n$ is given by:
	\[
	R_{\mathscr{S}^*_{\wp,n}}(\mathcal{F}_n) = (\sqrt{1+n^2e^2} - ne )^{1/n}.
	\]
\end{theorem}

\begin{proof}
	If $f \in \mathcal{F}_n$, then the function $h(z):=f(z)/z \in \mathcal{P}_n$ and
	\[
	\frac{zf'(z)}{f(z)} - 1 = \frac{zh'(z)}{h(z)}.
	\]
	Using Lemmas \ref{disk_lem} and \ref{p-nAlpha_lem}, we get
	\[
	\left| \frac{zf'(z)}{f(z)} -1 \right| =\biggl|\frac{zh'(z)}{h(z)}\biggl| \leq \frac{2nr^n}{1-r^{2n}} \leq \frac{1}{e}.
	\]
	Upon simplifying the last inequality, we get $r^{2n} + 2ner^n -1 \leq 0$. Thus, the $\mathscr{S}^*_{\wp,n}$-radius of $\mathcal{F}_n$ is the least positive root of $r^{2n} + 2ner^n -1=0$ in $(0,1)$. Since for the function $f_0(z) = z(1+z^n)/(1-z^n)$, $\RE(f_0(z)/z)>0$ in $\mathbb{D}$. We have   $f_0 \in \mathcal{F}_n$ and $zf'_0(z)/f_0(z) = 1 + 2nz^n/(1-z^{2n})$. Moreover, the result is sharp as we have at $z = R_{\mathscr{S}^*_{\wp,n}}(\mathcal{F}_n)$:
	\[
	\frac{zf'_0(z)}{f_0(z)} -1 = \frac{2nz^n}{1-z^{2n}} = \frac{1}{e}.
	\]
	This completes the proof. \qed
\end{proof}

Let $\mathcal{F}:=\mathcal{F}_1$, which is $\mathcal{F}:=\{ f \in \mathcal{A} : f(z)/z \in \mathcal{P}\}$.
MacGregor \cite{macgreg} showed that $r_0=\sqrt{2}-1$ is the radius of univalence and starlikeness for the class $\mathcal{F}$, we here below provide the  $\mathscr{S}^*_\wp$-radius for the same:
\begin{corollary} The $\mathscr{S}^*_\wp$-radius of the class $\mathcal{F}$ is given by
	\[
	R_{\mathscr{S}^*_\wp}(\mathcal{F}) = \sqrt{1+e^2}-e \approx 0.178105.
	\]
\end{corollary}


\begin{theorem}
	The sharp $\mathscr{S}^*_{\wp,n}$-radius of the class $\mathcal{CS}_n(\alpha)$ is given by
	
	$$	R_{\mathscr{S}^*_{\wp,n}}(\mathcal{CS}_n(\alpha))=\left( \frac{1/e}{\sqrt{(1+n-\alpha)^2 - (1/e)(2(1-\alpha)-1/e)} +1+n-\alpha} \right)^{1/n}. $$
\end{theorem}

\begin{proof}
	Let $f \in \mathcal{CS}_n(\alpha)$ and $g \in \mathcal{S}^*_n(\alpha)$. Then, we have $h(z):=f(z)/g(z) \in \mathcal{P}_n,$ which implies:
	\[
	\frac{zf'(z)}{f(z)} = \frac{zg'(z)}{g(z)} + \frac{zh'(z)}{h(z)}.
	\]
	Using Lemma \ref{p-nAlpha_lem} with $\alpha=0$ and Lemma \ref{p-nAB_lem}, we have
	\begin{equation}
	\label{eq_CSn-1}
	\left| \frac{zf'(z)}{f(z)} -a\right| \leq \frac{2(1+n-\alpha)r^n}{1-r^{2n}},
	\end{equation}
	where $a:=(1+(1-2\alpha)r^{2n})/(1-r^{2n}) \geq 1$. Note that $a\leq1+(e-e^{-1})/2$ if and only if $r^{2n}\leq(e^2-1)/(e^2-1+4e(1-\alpha))$. Let $r\leq R_{\mathscr{S}^*_{\wp,n}}(\mathcal{CS}_n(\alpha)).$  Then
	\begin{align*}
	r^{2n} &\leq\left(\frac{1}{e(2-\alpha)+\sqrt{e^2(2-\alpha)^2-e(2-2\alpha-\frac{1}{e})}}\right)^2\\
	&\leq \frac{1}{2e^2\alpha^2-(8e^2-2e)\alpha+(8e^2-2e+1)}.
	\end{align*}
	Further, the expression on the right of the above inequality is less than or equal to $\frac{e^2-1}{(e^2-1)+4e(1-\alpha)},$  provided
	$$T(\alpha):=2e^2(e^2-1)\alpha^2 - (8e^4-2e^3-8e^2-2e) \alpha + (8e^4-2e^3-8e^2-2e)  \geq0.$$
	Since $T'(\alpha)<0$ and $\min_{0<\alpha<1}{T(\alpha)}=\lim_{\alpha\rightarrow1}T(\alpha)=2e^2(e^2-1) >0.$
	Therefore, $a\leq 1+(e-e^{-1})/2$. Using Lemma \ref{disk_lem}, it follows that the disk, given by (\ref{eq_CSn-1}) is contained in the cardioid $\wp(\mathbb{D})$, if
	$$\frac{1 - 2(1+n-\alpha)r^n + (1-2\alpha)r^{2n}}{1-r^{2n}} \geq 1 - \frac{1}{e},$$
	which is equivalent to
	$(2-2\alpha-1/e)r^{2n} - 2(1+n-\alpha)r^n + 1/e \geq 0,$
	and holds when $r\leq R_{\mathscr{S}^*_{\wp,n}}(\mathcal{CS}_n(\alpha)) $. For sharpness, we consider the following functions
	\begin{equation}
	\label{eq_CSn-2}
	f_0(z) := \frac{z(1+z^n)}{(1-z^n)^{(n+2-2\alpha)/n}} \quad \text{and} \quad g_0(z) := \frac{z}{(1-z^n)^{2(1-\alpha)/n}},
	\end{equation}
	such that $f_0(z)/g_0(z) = (1+z^n)/(1-z^n)$ and $zg'_0(z)/g_0(z) = (1+(1-2\alpha)z^n)/(1-z^n)$. Moreover, $\RE(f_0(z)/g_0(z))>0$ and $\RE(zg'_0(z)/g_0(z))>\alpha$ in $\mathbb{D}$. Hence $f_0 \in \mathcal{CS}_n(\alpha)$ and
	\[
	\frac{zf'_0(z)}{f_0(z)} = \frac{1 + 2(1+n-\alpha)z^n + (1-2\alpha)z^{2n}}{1-z^{2n}}.
	\]
	For $z = R_1 e^{i\pi/n}$, we have $zf'_0(z)/f_0(z) = 1-1/e$. \qed
\end{proof}

\begin{theorem}
	\label{Scn-rad-thm-end}
	The $\mathscr{S}^*_{\wp,n}$-radius of the class $\mathcal{S}^*_n[A,B]$ is given by
	\begin{itemize}
		\item[$(i)$]
		
		$R_{\mathscr{S}^*_{\wp,n}}(\mathcal{S}^*_n[A,B]) = \min \biggl\{1,\; \biggl(\frac{1/e}{A-(1-1/e)B}\biggl)^{\frac{1}{n}} \biggl\},\;\text{when}\;  0\leq B < A \leq 1.$
		
		\item[$(ii)$]
		
		$   R_{\mathscr{S}^*_{\wp,n}}(\mathcal{S}^*_n[A,B]) = \left\{
		\begin{array}{ll}
		R_1, & \text{if}\; R_1 \leq r_1\\
		R_2, & \text{if}\; R_1 > r_1	
		\end{array}	
		\right.
		\;\text{when}\; -1 \leq B < 0 \leq A \leq 1.$
		
		where $$R_1:= R_{\mathscr{S}^*_{\wp,n}}(\mathcal{S}^*_n[A,B]) \;\text {as defined in part $($i$)$},$$
		$$R_2=\min\{1, (e/(A-(e+1)B))^{1/n}\}$$
		and
		$$ r_1=\left(\frac{(e^2-1)/{2e}}{({(e^2+2e-1)}/{2e})B^2-AB}\right)^{1/2n}.$$
		
	\end{itemize}
	In particular, for the class $\mathcal{S}^*$, we have $R_{\mathscr{S}^*_\wp}(\mathcal{S}^*)=1/(2e-1)$.
\end{theorem}

\begin{proof}
	Let $f \in \mathcal{S}^*_n[A,B]$. Using Lemma \ref{p-nAB_lem}, we have
	\begin{equation}
	\label{eq_S*nAB}
	\left| \frac{zf'(z)}{f(z)} -\frac{1-ABr^{2n}}{1-B^2r^{2n}} \right| \leq \frac{(A-B)r^n}{1-B^2r^{2n}}.
	\end{equation}
	{\bf(i)} If $0\leq B< A\leq1$, then
	
	$$a := \frac{1-ABr^{2n}}{1-B^2r^{2n}}\leq 1.$$
	Further, by Lemma \ref{disk_lem} and equation \eqref{eq_S*nAB}, we see that $f \in \mathscr{S}^*_{\wp,n}$ if
	$$\frac{ABr^{2n} + (A-B)r^n -1}{1-B^2 r^{2n}} \leq \frac{1}{e} - 1,$$
	which upon simplification, yields
	$$r\leq \biggl(\frac{1/e}{A-(1-1/e)B}\biggl)^{1/n}.$$
	The result is sharp due to the function $f_0(z),$ given by
	\begin{equation}\label{Sn_ab_1}
	f_0(z) =
	\left\{
	\begin{array}{ll}
	z(1+Bz^n)^{\frac{A-B}{nB}}; & B\neq0, \\
	z\exp\left(\frac{Az^n}{n}\right);    &  B=0.
	\end{array}	
	\right.
	\end{equation}
	{\bf(ii)}	If $-1 \leq B < 0 < A \leq 1$, then
	$$a := \frac{1-ABr^{2n}}{1-B^2r^{2n}}\geq 1.$$
	Let us first assume that $R_1\leq r_1$. Note that $r\leq r_1$ if and only if $a\leq1+(e-e^{-1})/2$. In particular, for $0\leq r\leq R_1$, we have $a\leq 1+(e-e^{-1})/2$. Further, from Lemma \ref{disk_lem}, we have $f\in \mathscr{S}^*_{\wp,n}$ in $|z|\leq r$, if
	$$\frac{(A-B)r^n}{1-B^2r^{2n}}\leq (a-1)+\frac{1}{e},$$
	which holds whenever $r\leq R_1$.
	Let us now assume that $R_1>r_1$. Thus $r\geq r_1$ if and only if $a\geq 1+(e-e^{-1})/2$. In particular, for $r\geq R_1$, we have $a\geq 1+(e-e^{-1})/2$. Further, from Lemma \ref{disk_lem}, we have $f\in \mathscr{S}^*_{\wp,n}$ in $|z|\leq r$ whenever
	$$\frac{(A-B)r^n}{1-B^2r^{2n}}\leq e-(a-1),$$
	which holds when $r\leq R_2$. Hence, the result follows with sharpness due to $f_0(z)$ given in (\ref{Sn_ab_1}). \qed
\end{proof}

\begin{theorem}
	The $\mathscr{S}^*_{\wp,n}$-radius of the class $\mathcal{M}^*_n(\beta)$, $(\beta>1)$ is given by
	$$	R_{\mathscr{S}^*_{\wp,n}}(\mathcal{M}^*_n(\beta))=(2e(\beta-1)+1)^{-1/n}.$$
\end{theorem}
\begin{proof}
	If $f\in\mathcal{M}_n(\beta)$, then $zf'(z)/f(z) \prec (1+(2\beta-1)z)/(1+z)$. Now using Lemma \ref{p-nAB_lem}, we have
	$$\left|\frac{zf'(z)}{f(z)}-\frac{1+(1-2\beta)r^{2n}}{1-r^{2n}}\right|\leq\frac{(\beta-1)2r^n}{1-r^{2n}}.$$
	Note that for $\beta>1$, we have $(1+(1-2\beta)r^{2n})/(1-r^{2n})<1$. Therefore, by Lemma \ref{disk_lem}, we get $f\in\mathscr{S}^*_{\wp,n}$ in $|z|<r$, provided
	$$\frac{(\beta-1)2r^n}{1-r^{2n}}-\frac{1+(1-2\beta)r^{2n}}{1-r^{2n}}\leq \frac{1}{e}-1,$$
	which holds whenever $r\leq R_{\mathscr{S}^*_{\wp,n}}(\mathcal{M}^*_n(\beta))$. The result is sharp due to $$f_0(z):=\frac{z}{(1-z^n)^{{2(1-\beta)}/{n}}},$$
	as we see that $zf_0'(z)/f_0(z)=(1+(1-2\beta)z^n)/(1-z^n)=1-1/e$ when $z=R_{\mathscr{S}^*_{\wp,n}}(\mathcal{M}^*_n(\beta))$. \qed
\end{proof}

In the following theorem, we attempt to find   the sharp $\mathscr{S}^*_\wp$-radii of the class $\mathcal{S}^*(\psi)$, for different choices of $\psi$ such as $1+\sin(z)$, $\sqrt{2}-c\sqrt{(1-z)/(1+2cz)}$, $1+4z/3+2z^2/3$, $z+\sqrt{1+z^2}$,  $e^z$ and $\sqrt{1+z}$, where $c:=\sqrt{2}-1$. Authors in \cite{sinefun,mendi,mendi2exp,raina,naveen14,sokol96} introduced and studied these subclasses of starlike functions which we denote by $\mathcal{S}^*_{s}$, $\mathcal{S}^*_{RL}$, $\mathcal{S}^*_{e}$, $\Delta^*$, $\mathcal{S}^*_{C}$ and  $\mathcal{S}^*_{L}$, respectively.

\begin{theorem}
	The sharp $\mathscr{S}^*_\wp$-radii of $\mathcal{S}^*_{L}, \mathcal{S}^*_{RL}, \mathcal{S}^*_{e}, \mathcal{S}^*_{C}, \mathcal{S}^*_{s} \, and \, \Delta^*$ are:
	\begin{itemize}
		\item [$(i)$] $R_{\mathscr{S}^*_\wp}(\mathcal{S}^*_{L}) =(2e-1)/e^2 \approx 0.600423$.
		
		\item [$(ii)$] $R_{\mathscr{S}^*_\wp}(\mathcal{S}^*_{RL}) = \tfrac{1+2(\sqrt{2}-1)e}{e^2(\sqrt{2}-1)(\sqrt{2}-1 + 2(\sqrt{2}-1+e^{-1})^2)} \approx 0.648826$.
		
		\item [$(iii)$]  $R_{\mathscr{S}^*_\wp}(\mathcal{S}^*_{e}) = 1-\ln(e-1) \approx 0.458675$.
		
		\item [$(iv)$]$R_{\mathscr{S}^*_\wp}(\mathcal{S}^*_{C}) =  1-\sqrt{1-3/2e} \approx 0.330536$.
		
		\item [$(v)$] $R_{\mathscr{S}^*_\wp}(\mathcal{S}^*_{s}) = \arcsin(1/e) \approx 0.376727$.
		
		\item [$(vi)$] $R_{\mathscr{S}^*_\wp}(\Delta^*) = (2e-1)/(2e(e-1)) \approx 0.474928$.
	\end{itemize}
\end{theorem}

\begin{proof}
	\begin{itemize}
		\item [$(i)$]
		If $f \in \mathcal{S}^*_{L}$, then for $|z|=r$
		\[
		\left|\frac{zf'(z)}{f(z)} -1 \right| \leq 1 - \sqrt{1-r} \leq \frac{1}{e},
		\]
		provided $r \leq (2e-1)/e^2 =: R_{\mathscr{S}^*_\wp}(\mathcal{S}^*_{L})$. Now consider the function
		\[
		f_0(z) := \frac{4z\exp(2(\sqrt{1+z}-1))}{(1+\sqrt{1+z})^2}.
		\]
		Since $zf_0'(z)/f_0(z) = \sqrt{1+z}$, it follows that $f_0 \in \mathcal{S}^*_{L}$ and for $z=-R_{\mathscr{S}^*_\wp}(\mathcal{S}^*_{L})$, we get $zf_0'(z)/f_0(z)=1-1/e$. Hence the result is sharp.
		
		\item [$(ii)$]
		Let $f \in \mathcal{S}^*_{RL}$. Then for $|z|<r$,
		\[
		\left|\frac{zf'(z)}{f(z)}-1 \right| \leq 1 - \sqrt{2} + (\sqrt{2}-1) \sqrt{\frac{1+r}{(1-2(\sqrt{2}-1)r)}} \leq \frac{1}{e},
		\]
		provided
		$$
		r \leq \frac{1+2(\sqrt{2}-1)e}{e^2(\sqrt{2}-1)(\sqrt{2}-1 + 2(\sqrt{2}-1+e^{-1})^2)} =: R_{\mathscr{S}^*_\wp}(\mathcal{S}^*_{RL}).
		$$
		For the sharpness, consider
		\[
		f_0(z):= z \exp\left(\int^z_0\frac{q_0(t)-1}{t}dt \right),
		\]
		where
		\[
		q_0(z) = \sqrt{2} - (\sqrt{2}-1) \sqrt{\frac{1-z}{(1+2(\sqrt{2}-1)z)}}.
		\]
		Now for $z=-R_{\mathscr{S}^*_\wp}(\mathcal{S}^*_{RL})$, we have
		\[
		\frac{zf_0'(z)}{f_0(z)} = \sqrt{2} - (\sqrt{2}-1) \sqrt{\frac{1-z}{(1+2(\sqrt{2}-1)z)}} = 1-\frac{1}{e}.
		\]
		\item [$(iii)$]
		Let $\rho=1-\ln(e-1)$, $q(z)=e^z$ and $f\in \mathcal{S}^*_{e} $. To prove that $f\in\mathscr{S}^*_\wp$ in $|z|<\rho$, it only suffices to show that $q(\rho z) \prec \wp(z)$ for $z\in\mathbb{D}$ and thus for $|z|=r$, we must have $e^{-r}\geq \wp(-1),$
		which gives $r\leq\rho.$ Now the difference of the square of the radial distances  from the point $(1,0)$ to the corresponding points on the boundary curves  $\partial\wp(e^{i\theta})$ and $\partial q(\rho e^{i\theta})$ is given by
		\[
		T(\theta):=e^{2\cos\theta}-e^{2\rho\cos\theta}-1+2e^{\rho\cos\theta}\cos(\rho\sin\theta)\quad(0\leq\theta\leq\pi).
		\]
		Since $T'(\theta)\leq0$ and $T(0)>0$, it follows that the condition $r\leq\rho$ is also sufficient for $q(\rho z)=e^{\rho z} \prec 1+ze^z=\wp(z)$.
		For the sharpness, consider the function
		\[
		f_0(z):= z \exp\left(\int^z_0\frac{e^t -1}{t}dt \right).
		\]
		Since $zf_0'(z)/f_0(z) = e^z$, which implies $f_0 \in \mathcal{S}^*_{e}$ and for $z=-R_{\mathscr{S}^*_\wp}(\mathcal{S}^*_{e})$, we have $e^z = 1-1/e$.

		\item [$(iv)$]
		Let $\rho=1-\sqrt{1-3/2e}$, $q(z)=1+4z/3+2z^2/3$ and $f\in \mathcal{S}^*_{C} $. To prove that $f\in\mathscr{S}^*_\wp$ in $|z|<\rho$, we make use of the fact that for $|z|=r<1$, the minimum distance of $w=q(z)$ from $1$ must be less than $1/e$. Therefore, for $f\in\mathscr{S}^*_\wp$ in $|z|<r$, it is necessary that
		$$\frac{4}{3}r-\frac{2}{3}r^2\leq\frac{1}{e},$$
		which gives $r\leq\rho$. Since $r_0=1/2 (>\rho)$ is the radius of convexity of $q(z)$ and it is symmetric about the real axis, we have $q(-r)\leq \RE q(z) \leq q(r)$ for $|z|=r<1/2$. Thus $1-1/e\leq \RE q(\rho z)\leq 1+4\rho/3+2\rho^2/3<1+e$. Therefore, to prove that $q(\rho z)\prec\wp(z)$, it suffices to show that $\max_{|z|=1}|\arg(q(\rho z)|\leq \max_{|z|=1}|\arg(\wp(z)|$. Since
		\begin{align*}
		\max_{0\leq\theta\leq\pi}\arg(q(\rho e^{i\theta})) &= \max_{0\leq\theta\leq\pi}\arctan\left(\frac{\frac{4}{3}\rho\sin{\theta}+\frac{2}{3}{\rho}^2\sin{2\theta}}{1+\frac{4}{3}\rho\cos{\theta}+\frac{2}{3}{\rho}^2\cos{2\theta}}\right)\\
		&\leq \arctan\left(\frac{\frac{4}{3}\rho\sin{\theta_0}+\frac{2}{3}{\rho}^2\sin{2\theta_0}}{1-\frac{4}{3}\rho+\frac{2}{3}{\rho}^2}\right)\\
		&\approx(0.401955)\pi/2 < \max_{0\leq\theta\leq\pi} \arg{\wp(e^{i\theta})}\approx(0.89782)\pi/2,
		\end{align*}
		where $\theta_0\in (0,\pi)$ is the only root of $\cos\theta+\rho\cos{2\theta}=0$. Hence, the condition $r\leq\rho$ is also sufficient for $q(\rho z)\prec\wp(z)$. Let us consider the function
		$$f_0(z) := z \exp\biggl(\frac{4z+z^2}{3}\biggl).$$
		Since $zf_0'(z)/f_0(z) = q(z)$, it follows that $f_0 \in \mathcal{S}^*_{C}$ and for $z=-R_{\mathscr{S}^*_\wp}(\mathcal{S}^*_{C})$, we get $q(z) = 1-1/e$. Hence the result is sharp.
		
		\item [$(v)$]
		Let $\rho=\sin^{-1}(1/e)$, $q(z)=1+\sin(z)$ and $f\in \mathcal{S}^*_{s}$. Then using the similar argument as in part (iv) together with a result (\cite{sinefun}, Theorem 3.3) and Lemma \ref{disk_lem}, we have $f\in\mathscr{S}^*_\wp$ in $|z|<r$, provided
		$\sin{r}\leq1/e$ which in turn gives $r\leq\rho$.
		For the sharpness, we consider the function
		\[
		f_0(z):= z \exp\left(\int^z_0\frac{\sin t}{t}dt \right).
		\]
		Since $zf_0'(z)/f_0(z) = q(z)$, we have $f_0 \in \mathcal{S}^*_{s}.$ For $z=-R_{\mathscr{S}^*_\wp}(\mathcal{S}^*_{s})$, we arrive at $q(z) = 1-1/e$.
		
		\item [$(vi)$]
		Let $\rho=(2e-1)/(2e(e-1))$, $q(z)=z+\sqrt{1+z^2}$ and $f\in \Delta^*$. Clearly  $\min|z+\sqrt{1+z^2}-1|=1+r-\sqrt{1+r^2}$ whenever $|z|=r<1$. Therefore, using Lemma \ref{disk_lem}, for $f\in\mathscr{S}^*_\wp$, we must have $\sqrt{1+r^2}-r\leq1-1/e$, which gives $r\leq \rho$. Following the similar argument as in part $(iv)$, we see that $r\leq \rho$ is also sufficient condition for $q(\rho z)\prec \wp(z)$ to hold. Now for the function
		\[
		f_0(z):= z \exp\left(\int^z_0\frac{t+\sqrt{1+t^2}-1}{t}dt \right),
		\]
		we have $zf_0'(z)/f_0(z) = q(z)$, which implies  $f_0 \in \Delta^*.$ For $z=-R_{\mathscr{S}^*_\wp}(\Delta^*)$, we get $q(z) = 1-1/e$, which shows that the result is sharp. \qed
	\end{itemize}
\end{proof}

Now for our next radius problem, we need to consider some classes: Here below we presume, the value of $\alpha$ to be either $0$ or $1/2$.
\begin{equation*}
\mathcal{F}_1(\alpha) := \biggl\{ f \in \mathcal{A}_n : \RE \frac{f(z)}{g(z)} > 0 \;\text{and}\; \RE \frac{g(z)}{z} > \alpha,\; g \in \mathcal{A}_n  \biggl\},
\end{equation*}


\begin{equation*}
\mathcal{F}_2 := \biggl\{ f \in \mathcal{A}_n : \left|\frac{f(z)}{g(z)} -1 \right| < 1 \;\text{and}\; \RE \frac{g(z)}{z} > 0,\; g \in \mathcal{A}_n \biggl\}
\end{equation*}
and
\begin{equation*}
\mathcal{F}_3 := \biggl\{f \in \mathcal{A}_n : \left|\frac{f(z)}{g(z)} -1 \right| < 1  \;\text{and}\; g\in\mathcal{C}, \; g \in \mathcal{A}_n \biggl\}.
\end{equation*}

\begin{theorem} \label{ratio_func}
	The sharp $\mathscr{S}^*_{\wp,n}$-radii of functions in the classes $\mathcal{F}_1(\alpha),\, \mathcal{F}_2$ and $\mathcal{F}_3$, respectively, are:
	\begin{itemize}
		\item [$(i)$] $R_{\mathscr{S}^*_{\wp,n}}(\mathcal{F}_1(0)) = \left(\sqrt{4n^2e^2+1}-2ne\right)^{1/n}$.
		
		\item [$(ii)$] $R_{\mathscr{S}^*_{\wp,n}}(\mathcal{F}_1(1/2)) = \left(2/(\sqrt{(3ne+2)^2-8ne}+3ne)\right)^{1/n}$.
		
		\item [$(iii)$] $R_{\mathscr{S}^*_{\wp,n}}(\mathcal{F}_2) = \left(2/(\sqrt{(3ne+2)^2-8ne}+3ne)\right)^{1/n}$.
		
		\item [$(iv)$] $R_{\mathscr{S}^*_{\wp,n}}(\mathcal{F}_3) = \left( \frac{ \sqrt{(n+1)^2 + 4(n-1+1/e)/e} -(1+n)}{2(n-1+1/e)} \right)^{1/n}$.
	\end{itemize}
\end{theorem}

\begin{proof}
	Let us consider the functions $p,h: \mathbb{D} \rightarrow \mathbb{C},$ defined by $p(z)=g(z)/z$ and $h(z)=f(z)/g(z)$. We write $p_0(z)=g_0(z)/z$ and $h_0(z)=f_0(z)/g_0(z)$.	
	\begin{itemize}
		\item [$(i)$] If $f \in \mathcal{F}_1(0)$, then $p,h \in \mathcal{P}_n$ such that $f(z)=zp(z)h(z)$. Thus it follows from Lemma \ref{p-nAlpha_lem} that
		\[
		\left|\frac{zf'(z)}{f(z)} -1 \right| \leq \frac{4nr^n}{1-r^{2n}} \leq \frac{1}{e},
		\]
		provided $r \leq \left(\sqrt{4n^2e^2+1}-2ne\right)^{1/n} =: R_{\mathscr{S}^*_{\wp,n}}(\mathcal{F}_1(0))$. Thus $f \in \mathscr{S}^*_{\wp,n}$ whenever $r \leq R_{\mathscr{S}^*_{\wp,n}}(\mathcal{F}_1(0))$.
		Now for the functions
		\[
		f_0(z)= z \left(\frac{1+z^n}{1-z^n}\right)^2 \quad \text{and} \quad g_0(z) = z \left(\frac{1+z^n}{1-z^n}\right),
		\]
		we have, $\RE{h_0(z)}>0$ and $\RE{p_0(z)}>0$. Hence $f_0 \in \mathcal{F}_1(0)$. For $z = R_{\mathscr{S}^*_{\wp,n}}(\mathcal{F}_1(0))e^{i\pi/n},$
		we see that
		\[
		\frac{zf'_0(z)}{f_0(z)} = 1 + \frac{4nz^n}{1-z^{2n}} = 1 - \frac{1}{e}.
		\]
		Thus the result is sharp.
		\item [$(ii)$] Let $f \in \mathcal{F}_1(1/2)$. Then $h \in \mathcal{P}_n$ and $p \in \mathcal{P}_n(1/2)$. Since $f(z)=zp(z)h(z)$, it follows from Lemma \ref{p-nAlpha_lem} that
		\[
		\left|\frac{zf'(z)}{f(z)} -1 \right| \leq \frac{2nr^n}{1-r^{2n}} + \frac{nr^n}{1-r^n} = \frac{3nr^n + nr^{2n}}{1-r^{2n}} \leq \frac{1}{e},
		\]
		provided
		\[
		r \leq \left( \frac{\sqrt{9n^2e^2 + 4(ne+1)} -3ne}{2(ne+1)} \right)^{1/n} =: R_{\mathscr{S}^*_{\wp,n}}(\mathcal{F}_1(1/2)).
		\]
		Thus $f \in \mathscr{S}^*_{\wp,n}$ whenever $r \leq R_{\mathscr{S}^*_{\wp,n}}(\mathcal{F}_1(1/2))$.
		For the functions
		\[
		f_0(z) = \frac{z(1+z^n)}{(1-z^n)^2} \quad \text{and} \quad g_0(z) = \frac{z}{1-z^n},
		\]
		we have, $\RE{h_0(z)}>0$ and $\RE{p_0(z)}>1/2$. Hence $f \in \mathcal{F}_1(1/2)$. The result is sharp, since for $z = R_{\mathscr{S}^*_{\wp,n}}(\mathcal{F}_1(1/2))$, we have
		\[
		\frac{zf'_0(z)}{f_0(z)} -1 = \frac{3nz^n + nz^{2n}}{1-z^{2n}} = \frac{1}{e}.
		\]
		
		\item [$(iii)$] Let $f \in \mathcal{F}_2$. Then $p \in \mathcal{P}_n$. Since $|h(z)-1|<1$ if and only if $\RE(1/h(z))>1/2$. Therefore, $1/h \in \mathcal{P}_n(1/2)$. Since $f(z)/h(z)=zp(z)$, using Lemma \ref{p-nAlpha_lem}, we have
		\[
		\left|\frac{zf'(z)}{f(z)} -1 \right| \leq \frac{3nr^n + nr^{2n}}{1-r^{2n}} \leq \frac{1}{e},
		\]
		provided $r^n\leq 2/(\sqrt{(3ne+2)^2-8ne}+3ne)$. For the sharpness, consider
		\[
		f_0(z) := \frac{z(1+z^n)^2}{1-z^n} \quad \text{and} \quad g_0(z) := \frac{z(1+z^n)}{1-z^n}.
		\]
		Since
		$$
		|h_0(z)-1|=|z^n|<1 \; \text{and} \; \RE{p_0(z)} = \RE \frac{1+z^n}{1-z^n} > 0.
		$$
		Therefore, $f_0 \in \mathcal{F}_2$ and for $z = R_{\mathscr{S}^*_{\wp,n}}(\mathcal{F}_2) e^{i\pi/n}$, we have
		\[
		\left|\frac{zf'_0(z)}{f_0(z)} -1 \right|=\left| \frac{3nz^n - nz^{2n}}{1-z^{2n}}\right| = \frac{1}{e}.
		\]
		
		\item [$(iv)$] Let $f \in \mathcal{F}_3$. Then $1/h(z)=g(z)/f(z)\in \mathcal{P}_n(1/2)$ and
		\begin{equation}\label{F4}
		\frac{zf'(z)}{f(z)}=\frac{zg'(z)}{g(z)}-\frac{zh'(z)}{h(z)}.
		\end{equation}
		Using a result due to Marx-Strohh\"{a}cker that every convex function is starlike of order $1/2$, it follows from Lemma \ref{p-nAB_lem} that
		\begin{equation}
		\label{ratio_fconv_eqn}
		\left|\frac{zg'(z)}{g(z)} - \frac{1}{1-r^{2n}}\right| \leq \frac{r^n}{1-r^{2n}}.
		\end{equation}
		Now using Lemma \ref{p-nAlpha_lem} and equation \eqref{ratio_fconv_eqn}, we have
		\[
		\left|\frac{zf'(z)}{f(z)} - \frac{1}{1-r^{2n}}\right| \leq \frac{r^n}{1-r^{2n}} + \frac{nr^n}{1-r^n} = \frac{(n+1)r^n + nr^{2n}}{1-r^{2n}}.
		\]
		Thus using Lemma \ref{disk_lem}, we have $f \in \mathscr{S}^*_{\wp,n}$, provided
		\[
		\frac{(n+1)r^n + nr^{2n}}{1-r^{2n}} \leq \left(\frac{1}{1-r^{2n}}-1\right) + \frac{1}{e},
		\]
		which implies	$r \leq  R_{\mathscr{S}^*_{\wp,n}}(\mathcal{F}_3)$.	Now consider the functions
		\[
		f_0(z) = \frac{z(1+z^n)}{(1-z^n)^{1/n}} \quad \text{and} \quad g_0(z) = \frac{z}{(1-z^n)^{1/n}}.
		\]
		Since $g_0 \in \mathcal{C}$ and $|h_0(z) -1| = |z^n| < 1$. Therefore, $f_0 \in \mathcal{F}_3$ and for $z = R_{\mathscr{S}^*_{\wp,n}}(\mathcal{F}_3) e^{i\pi/n}$, we have
		$
		{zf'_0(z)}/{f_0(z)} = 1 - {1}/{e},
		$
		which confirms the sharpness of the result. \qed
	\end{itemize}
\end{proof}

\section{\bf Coefficient Problems}

The following lemmas are needed to prove our coefficient results.

\begin{lemma}\label{1} \emph{\cite{minda94}}
	Let $p\in\mathcal{P}$ be of the form \eqref{caratheodory}. Then for a complex number $\tau$, we have
	\begin{equation*}
	|p_2-\tau{p_1}^2|\leq2 \max(1, |2\tau-1|).
	\end{equation*}
\end{lemma}

\begin{lemma}\label{3} \emph{\cite{shelly}}
	Let $p\in\mathcal{P}$ be of the form \eqref{caratheodory}. Then for  $n,m\in\mathbb{N}$,
	\begin{equation*}
	|p_{n+m}-\gamma p_n p_m|\leq
	\left\{
	\begin{array}
	{lr}
	2, &  0\leq\gamma\leq1; \\
	2|2\gamma-1|,   &\text{elsewhere}.
	\end{array}
	\right.
	\end{equation*}
\end{lemma}

Here below, we  partially disclose the lemma given in \cite{ravi-pvalent}, which is required in sequel.
\begin{lemma}\label{5}\emph{\cite{ravi-pvalent}}
	If $\omega\in\Omega$ be of the form \eqref{schwarz_fn}, then
	\begin{equation*}
	|c_3+\mu c_1c_2+\nu c^3_1|\leq \Psi(\mu,\nu),
	\end{equation*}
	where
	\begin{equation*}
	\Psi(\mu,\nu)= \frac{2}{3}(|\mu|+1)\biggl(\frac{|\mu|+1}{3(1+\nu+|\mu|)}\biggl)^{\frac{1}{2}}\quad\text{for}\quad (\mu,\nu)\in D_8\cup D_9
	\end{equation*}
	and
	\begin{align*}
	D_8&:=\biggl\{(\mu,\nu):\frac{1}{2}\leq|\mu|\leq2,-\frac{2}{3}(|\mu|+1)\leq \nu\leq\frac{4}{27}(|\mu|+1)^3-(|\mu|+1)\biggl\}\\
	D_9&:=\biggl\{(\mu,\nu):|\mu|\geq2,-\frac{2}{3}(|\mu|+1)\leq \nu\leq \frac{2|\mu|(|\mu|+1)}{{\mu}^2+2|\mu|+4}\biggl\}.
	\end{align*}
\end{lemma}

The following lemma, carries the expression for $p_2$ and $p_3$  in terms of $p_1$, derived in \cite{R.Libera,Libera} and $p_4$ in terms of $p_1$ obtained in \cite{lem_p4}.
\begin{lemma}\label{2}
	Let $p\in\mathcal{P}$. Then for some complex numbers $\zeta$, $\eta$ and $\xi$ with $|\zeta|\leq1$, $|\eta|\leq1$ and $|\xi|\leq1$, we have
	\begin{align*}
	2p_2 &= p_1^2+\zeta(4-p_1^2),\\
	4p_3 &= p_1^3+2p_1\zeta(4-p_1^2)-p_1{\zeta}^2(4-p_1^2)+2(4-p_1^2)(1-|\zeta|^2)\eta\\
	\text{and}\quad	8p_4 &=p^4_1+(4-p^2_1)\zeta(p^2_1(\zeta^2-3\zeta+3)+4\zeta)\\
	&\quad-4(4-p^2_1)(1-|\zeta|^2)(p_1(\zeta-1)\eta+\bar{\zeta}{\eta}^2-(1-|\eta|^2)\xi).
	\end{align*}
\end{lemma}

We now define the function $f_n$ such that $f_n(0)=f'_n(0)-1=0$ and
\begin{equation*}
\frac{zf'_n(z)}{f_n(z)}=\wp(z^{n}) \quad\quad(n=1,2,3,\cdots),
\end{equation*}
which acts as an extremal function for many subsequent results and we have 
\begin{equation}\label{extremals}
f_n(z)=z \exp((e^{z^n}-1)/n)
\end{equation}

\begin{theorem}
	Let $f(z)= z+\sum_{k=2}^{\infty}b_kz^k \in \mathscr{S}^*_\wp$ and  $\alpha=(1+e)^2$, then
	\begin{equation}
	\sum_{k=2}^{\infty}(k^2-\alpha)|b_k|^2\leq \alpha-1.
	\end{equation}
\end{theorem}
\begin{proof}
	If $f\in\mathscr{S}^*_\wp$, then $zf'(z)/f(z)=\wp(\omega(z))$, where $\omega\in \Omega$. Now for $0\leq r<1$, we have
	\begin{align*}
	2\pi \sum_{k=1}^{\infty}k^2|b_k|^2r^{2k} &= \int_{0}^{2\pi}|r e^{i \theta} f'(r e^{i \theta})|^2 d\theta
	\\&=\int_{0}^{2\pi}|f(r e^{i \theta})+f(r e^{i \theta}) \omega(r e^{i \theta}) e^{ \omega(r e^{i \theta})}|^2 d\theta
	\\ &\leq \int_{0}^{2\pi}\biggl( |f(r e^{i \theta})|+ |f(r e^{i \theta}) \omega(r e^{i \theta}) e^{\omega(r e^{i \theta})}|  \biggl)^2 d\theta
	\\ &= \int_{0}^{2\pi}|f(r e^{i \theta})|^2 d\theta + \int_{0}^{2\pi}|f(r e^{i \theta}) \omega(r e^{i \theta}) e^{ \omega(r e^{ i\theta})}|^2 d\theta
	\\& \quad + 2\int_{0}^{2\pi}|f(r e^{i \theta})|^2 |\omega(r e^{i \theta}) e^{\omega(r e^{i \theta})}| d\theta
	\\ &\leq \int_{0}^{2\pi}|f(r e^{i \theta})|^2 d\theta+\int_{0}^{2\pi}|f(r e^{i \theta})e^{\omega(r e^{i \theta})}|^2 d\theta
	\\& \quad+ 2\int_{0}^{2\pi}|f(r e^{i \theta})|^2 |e^{\omega(e^{i \theta})}| d\theta
	\\&\leq \int_{0}^{2\pi}|f(r e^{i \theta})|^2 d\theta + e^{2r}\int_{0}^{2\pi}|f(r e^{i \theta})|^2 d \theta \\&\quad+ 2e^{r}\int_{0}^{2\pi}|f(r e^{i \theta})|^2 d \theta\\ &\leq 2\pi(1+e^r)^2 \sum_{k=1}^{\infty}|b_k|^2r^{2k},
	\end{align*}
	
	which finally yields
	\begin{equation*}
	\sum_{k=1}^{\infty}(k^{2}-(1+e^r)^2)|b_k|^2r^{2k} \leq0,\quad 0\leq r<1.
	\end{equation*}
	Letting $r \rightarrow 1^{-}$, we get the desired result. \qed
\end{proof}

\begin{corollary}
	Let $f(z)=z+\sum_{k=4}^{\infty}b_kz^k \in \mathscr{S}^*_{\wp}$ and  $\alpha=(1+e)^2$, then
	\begin{equation*}
	|b_k|\leq \sqrt{\frac{\alpha-1}{k^2-\alpha}}, \text{ for all } k\geq4.
	\end{equation*}
\end{corollary}

\begin{example}\label{example}
	\begin{enumerate} [$(i)$]
		\item $z/(1-Az)^2\in \mathscr{S}^*_\wp$ if and only if $|A|\leq1/(2e-1)$.
		\item $f(z)=z+b_kz^k\in \mathscr{S}^*_\wp$ if and only if $|b_k|\leq1/(e(k-1)+1)$, where $k\in\mathbb{N}-\{1\}$.
		\item $f(z)=z\exp(Az)\in\mathscr{S}^*_\wp$ if and only if $|A|\leq1/e$.
	\end{enumerate}
\end{example}
\begin{proof}
	(i) If $A=1$, then $z/(1-z)^2\not\in \mathscr{S}^*_{\wp}$, since $|f(z)|\leq e^{e-1}$ for $f\in \mathscr{S}^*_{\wp}$. Now let $K(z)=z/(1-Az)^2$. If $A\neq1$, then the disk
	\begin{equation}
	\label{coeff-disk-1}
	\left|w-\frac{1+|A|^2}{1-|A|^2}\right|<\frac{2|A|}{1-|A|^2},
	\end{equation}
	is the image of $\mathbb{D}$  under the bilinear transformation $w=zK'(z)/K(z)=(1+Az)/(1-Az)$ with diameter's end points $x_L:=(1-|A|)/(1+|A|)$ and $x_R:=(1+|A|)/(1-|A|)$. Now for the disk \eqref{coeff-disk-1} to be inside the cardioid $\wp(\mathbb{D})$, it is necessary that $x_L\geq1-1/e$ which gives $|A|\leq 1/(2e-1)$. Conversely, let $|A|\leq 1/(2e-1)$. Then we have
	\[
	a:=\frac{1+|A|^2}{1-|A|^2}\leq \frac{2e-e^{-1}+2}{2(e-1)} \quad\text{and}\quad r:=\frac{2|A|}{1-|A|^2}\leq\frac{2e-1}{2e(e-1)}.
	\]
	Since $r_a> r$, thus Lemma \ref{disk_lem} ensures that disk  $\{w: |w-a|<r\}\subset\wp(\mathbb{D})$. Hence, $K\in\mathscr{S}^*_{\wp}$.\\
	(ii) Since $zf'(z)/f(z)=(1+kb_kz^{k-1})/(1+b_kz^{k-1})$ maps  $\mathbb{D}$ onto the disk $\{w\in \mathbb{C}: |w-a|<r \}$, where
	\begin{equation*}
	a:=\frac{1-k|b_k|^2}{1-|b_k|^2}\quad\text{and}\quad r:=\frac{(k-1)|b_k|}{1-|b_k|^2}.
	\end{equation*}
	Further $f(z)=z+b_kz^k\in \mathcal{S}^*$ if and only if $|b_k|\leq1/k$, which ensures $(1-k|b_k|^2)/(1-|b_k|^2)\leq1$. Therefore, in view of Lemma \ref{disk_lem}, $\{w\in \mathbb{C}: |w-a|<r \} \subset \wp(\mathbb{D})$ if and only if
	\begin{equation*}
	\frac{(k-1)|b_k|}{1-|b_k|^2}\leq\frac{1-k|b_k|^2}{1-|b_k|^2}-1+\frac{1}{e},
	\end{equation*}
	which is equivalent to $(ke-e+1)|b_k|^2+(ke-e)|b_k|-1\leq0$. Hence, $|b_k|\leq1/(e(k-1)+1)$.\\
	(iii) Since $zf'(z)/f(z)=1+Az$ maps $\mathbb{D}$ onto the disk $\{w\in \mathbb{C}: |w-1|<|A| \}$. Therefore, in view of Lemma \ref{disk_lem}, the inequality
	$$|w-1|<|A|\leq1/e,$$
	yields the necessary and sufficient condition $|A|\leq1/e$ for $1+Az\prec\wp(z)$. \qed
\end{proof}

\begin{remark}
	Note that when $k=\sqrt{2}+1$, we have $q_0(z)=1+(z/k)((k+z)/(k-z)) \prec \wp(z)$. Therefore, the class of starlike functions $\mathcal{S}^*(q_0)$ introduced in \cite{sushil} is contained in  $\mathscr{S}^*_\wp$. Further the sharp $\mathcal{S}^*(\psi)$-radius for the class $\mathcal{S}^*$ is also given by the relation
	\begin{equation}\label{s-radius}
	R_{\mathcal{S}^*(\psi)}(\mathcal{S}^*)=\max|A|,
	\end{equation}
	where $A$ is defined in such a way that $z/(1-Az)^2\in \mathcal{S}^*(\psi)$.
	Thus  if $z/(1-Az)^2\in \mathcal{S}^*(q_0) \subset \mathscr{S}^*_\wp $, then by Example \ref{example}, we see that $|A|\leq1/(2e-1)$ and therefore, in view of \eqref{s-radius}, we now state a result (\cite{sushil}, theorem 2.3, pg 203)  in its correct  form using the result (\cite{sushil}, theorem 3.2, pg 206):
	$$z/(1-Az)^2\in \mathcal{S}^*(q_0) \; \text{if and only if}\; |A|\leq \frac{3-2\sqrt{2}}{2\sqrt{2}-1} <\frac{1}{2e-1}.$$
	The authors proved that $|A|\leq1/3$. \qed
\end{remark}

\begin{theorem}\label{coefficients}
	Let $f(z)= z+\sum_{k=2}^{\infty}b_kz^k \in \mathscr{S}^*_{\wp}$, then
	\begin{equation}\label{coef}
	|b_2|\leq1,\quad |b_3|\leq1,\quad |b_4|\leq5/6\quad\text{and}\quad |b_5|\leq5/8.
	\end{equation}
	The bounds are sharp.
\end{theorem}
\begin{proof}
	Let $p(z)\in\mathcal{P}$. Since there exists one-one correspondence between the classes $\Omega$ and $\mathcal{P}$ via the following functions:
	\begin{equation*}
	\omega(z)=\frac{p(z)-1}{p(z)+1}\quad\text{and}\quad p(z)=\frac{1+\omega(z)}{1-\omega(z)}.
	\end{equation*}
	Therefore, for $f\in \mathscr{S}^*_{\wp}$, we have
	\begin{equation*}
	\frac{zf'(z)}{f(z)}=\wp(\omega(z))=\wp\biggl(\frac{p(z)-1}{p(z)+1}\biggl),
	\end{equation*}
	where
	\begin{equation} \label{z_1234}
	\wp\biggl(\frac{p(z)-1}{p(z)+1}\biggl) = 1+\frac{p_1}{2}z+\frac{p_2}{2}z^2+\biggl(-\frac{p_1^3}{16}+\frac{p_3}{2}\biggl)z^3 +\biggl(\frac{p_1^4}{24}-\frac{3p_1^2p_2}{16}+\frac{p_4}{2}\biggl)z^4+\cdots
	\end{equation}
	and
	\begin{align} \label{Z_1234}
	\frac{zf'(z)}{f(z)} &= 1+b_2z+(2b_3-b_2^2)z^2+(3b_4-3b_2b_3+b_2^3)z^3\nonumber\\  &\quad+(-b_2^4+2b_3(2b_2^2-b_3)-4b_2b_4+4b_5)z^4+\cdots.
	\end{align}
	On comparing the coefficients of $z^k$ $(k=1,2,3,4)$ in \eqref{z_1234} and \eqref{Z_1234}, we get
	\begin{align}\label{a15}
	b_2&=\frac{p_1}{2},\;
	b_3=\frac{1}{4}\biggl(p_2+\frac{p_1^2}{2}\biggl),\;
	b_4=\frac{1}{6}\biggl(p_3+\frac{3}{4}p_1p_2\biggl)\nonumber\\
	\text{and}\quad b_5&=\frac{1}{8}\biggl(\frac{1}{48}p^4_1+\frac{1}{4}p^2_2+\frac{2}{3}p_1p_3-\frac{1}{8}p^2_1p_2+p_4\biggl).
	\end{align}
	Now using the fact $|p_k|\leq2$, Lemma \ref{1}  with $\tau=-1/2$ and Lemma \ref{3} with $\gamma=-3/4$, we obtain $|b_2|\leq1$, $|b_3|\leq1$ and $|b_4|\leq5/6$ respectively.
	
	For $b_5$, using proper rearrangement of terms and then applying triangle inequality, we see that
	\begin{align*}
	|b_5|&=\frac{1}{8}\biggl|\frac{1}{48}p^4_1+\frac{1}{4}p^2_2+\frac{2}{3}p_1p_3-\frac{1}{8}p^2_1p_2+p_4\biggl|\\
	&=\frac{1}{8}\biggl|\frac{1}{48}p^4_1+(p_4+\frac{2}{3}p_1p_3)+\frac{1}{4}p_2(p_2-\frac{1}{2}p^2_1)\biggl|\\
	&\leq\frac{1}{8}\biggl(\frac{1}{48}|p_1|^4+|p_4+\frac{2}{3}p_1p_3|+\frac{1}{4}|p_2||p_2-\frac{1}{2}p^2_1|\biggl)\\
	&\leq\frac{1}{8}\biggl(\frac{1}{48}|p_1|^4-\frac{1}{4}|p_1|^2+\frac{4}{3}|p_1|+3\biggl)\\
	&=: G(p_1).
	\end{align*}
	Now to maximize the above expression, without loss of generality, we write
	$$G(p)=\frac{1}{48}p^4-\frac{1}{4}p^2+\frac{4}{3}p+3 \quad  (p\in[0,2]),$$
	then $G'(p)\geq 0$. Thus $G(p)\leq5$, which implies that $|b_5|\leq5/8$.
	The bounds for $b_k$ (k=1,2,3,4) are sharp with the extremal function $f_1$ defined in \eqref{extremals}. \qed
\end{proof}

Now in view of Theorem \ref{coefficients}, we conjecture the following:\\
\textbf{Conjecture.}
\textit{Let $f(z)\in \mathscr{S}^*_{\wp}$. Then the following sharp estimates hold:
	\begin{equation*}
	|b_k|\leq \frac{B_{k-1}}{(k-1)!} \quad\text{for all}\quad k\geq1,
	\end{equation*}
	where $B_k$ are Bell numbers satisfying the recurrence relation defined in \eqref{bell_rec} and the extremal function $f_1$ is given by (\ref{extremals})}.

\begin{remark}
	The logarithmic coefficients $d_k$ for $f\in \mathcal{S}$ are defined by the following series expansion:
	\begin{equation}\label{log-cof-eq}
	\log \frac{f(z)}{z}=2\sum_{k=1}^{\infty}d_kz^k,\quad z\in \mathbb{D}.
	\end{equation}
	Recently, Cho \cite{log_cho} obtained the sharp logarithmic coefficient bounds for the class $\mathcal{S}^*(\psi)$ given by \eqref{m-class}.
	Consequently, we have the following sharp result:\\
	\indent \textit{Let  $f\in\mathscr{S}^*_\wp$. Then the logarithmic coefficients of $f$ given by (\ref{log-cof-eq}) satisfies}
	\begin{equation*}
	|d_k|\leq {1}/{2}.
	\end{equation*}
\end{remark}
\begin{remark}
	Now if $f(z)\in \mathscr{S}^*_{\wp}$, then from \eqref{a15}, using triangle inequality together with Lemma \ref{1}, we obtain the  following estimates for the Fekete-Szeg\"{o} functional:
	\begin{equation}\label{feketo}
	|b_3-\mu b_2^2| = \frac{1}{4}\left|p_2-\left(\mu-\frac{1}{2}\right)p_1^2\right|  \leq\frac{1}{2}\max\left(1, 2|\mu-1|\right).
	\end{equation}
	Equality cases holds for the functions $f_1(z)=z \exp(e^{z}-1)$, when $\mu\in[1/2,3/2]$ and $f_2(z)=z \exp((e^{z^2}-1)/2)$, when $\mu\leq1/2$ or $\mu\geq3/2$ given by \eqref{extremals}. In particular for $\mu=1$, we have	$|H_2(1)|=|b_3-b_2^2|\leq{1}/{2}.$ 
\end{remark}

Now the Covering Theorem stated in Theorem \ref{GCR} ensures that for every $f$ in $\mathscr{S}^*_\wp$, $f(\mathbb{D})$ contains a disk of radius $e^{{1}/{e}-1}$	centered at the origin. Hence, every function $f\in\mathscr{S}^*_\wp$ has an inverse $f^{-1}$ which given by
\begin{equation*}\label{inversefun}
f^{-1}(w)=w+\sum_{k=2}^{\infty}A_kw^k=w-b_2w^2+(2b^2_2-b_3)w^3-(5b^3_2-5b_2b_3+b_4)w^4+\cdots,
\end{equation*}
then we have
$f^{-1}(f(z))=z$ and  $f(f^{-1}(w))=w$ for $|w|<r_0(f)$ and $r_0>e^{{1}/{e}-1}$.
Thus using Theorem \ref{coefficients} and equation \ref{feketo}, we easily obtain
\begin{equation*}
|A_2|\leq1 \quad\text{and}\quad |A_3|\leq1.
\end{equation*}
The bounds are sharp with extremal function $f^{-1}_1$, where $f_1$ is defined in (\ref{extremals}).

\begin{theorem}
	Let $f\in\mathscr{S}^*_\wp$. Then for $ f^{-1}(\omega)= \omega+\sum_{k=2}^{\infty}A_k\omega^k$, we have
	\begin{equation*}
	|A_4|\leq\frac{5}{6} \quad\text{and}\quad
	|A_3-\mu A_2^2|\leq
	\left\{
	\begin{array}
	{lr}
	3-\mu,     &  \mu\leq{5}/{2}; \\
	{1}/{2},   & {5}/{2}\leq\mu\leq{7}/{2}; \\
	\mu-3,     &  \mu\geq{7}/{2}.
	\end{array}
	\right.
	\end{equation*}
	The bounds are sharp.
\end{theorem}

\begin{proof}Consider the inverse function $f^{-1}(\omega)= \omega+\sum_{k=2}^{\infty}A_k\omega^k$, where we have $A_4=-5b^3_2+5b_2b_3-b_4$, which can be now rewritten in terms of Carath\'{e}odory coefficients using (\ref{a15}) as \begin{equation*} A_4=-\frac{1}{6}\left(p_3-3p_1p_2+\frac{15}{8}p^3_1\right).\end{equation*} Now using Lemma \ref{1} with $\tau=5/8$ and $|p_k|\leq2$,
	\begin{align*}
	|b_4|&=\frac{1}{6}\left|p_3-3p_1\left(p_2-\frac{5}{8}p^2_1\right)\right|
	\leq\frac{1}{6}(|p_3|+3|p_1||p_2-\frac{5}{8}p^2_1|)\leq\frac{5}{6}.
	\end{align*}
	The bound is sharp with extremal function $f^{-1}_1$, where $f_1$ is defined in (\ref{extremals}).
	Now for the Fekete-Szeg\"{o} type inequality for the inverse function $f^{-1}$, we have
	\begin{equation*}
	|A_3-\mu A^2_2|=|b_3-tb^2_2|, \quad t=\mu-2.
	\end{equation*}
	Thus using \eqref{feketo}, the desired sharp result follows. \qed
\end{proof}

\begin{theorem}\label{bound}
	Let $f(z)= z+\sum_{k=2}^{\infty}b_kz^k \in \mathscr{S}^*_{\wp}$, then
	\begin{equation*}
	|b_2b_3-b_4|\leq \frac{2}{3}\sqrt{\frac{2}{5}}.
	\end{equation*}
	The bound is sharp.
\end{theorem}
\begin{proof}
	Let $f\in \mathscr{S}^*_\wp$. Then
	\begin{equation}\label{bound1}
	\frac{zf'(z)}{f(z)}=\wp(\omega(z)),
	\end{equation}
	where $\omega\in\Omega$. Then  proceeding as in Theorem~\ref{coefficients}, from (\ref{bound1}), we have
	\begin{equation}\label{bound2}
	b_2=c_1,\quad b_2=\frac{1}{2}(c_2+2c^2_1) \quad\text{and}\quad b_4=\frac{1}{6}(2c_3+7c_1c_2+5c^3_1).
	\end{equation}
	Therefore, with $\mu=2, \nu=-1/2$ and $\psi(\mu,\nu)=|c_3+\mu c_1c_2+\nu c^3_1|$, we have
	\begin{equation*}
	|b_2b_3-b_4|=\frac{1}{3}|c_3+2c_1c_2-c^3_1/2|=\frac{1}{3}\psi(\mu,\nu).
	\end{equation*}
	Now using Lemma \ref{5}, we obtain
	\[
	|b_2b_3-b_4|\leq\frac{2}{3}\sqrt{\frac{2}{5}}.
	\]
	The bound is sharp  as there is an extremal function
	\begin{equation*}
	f(z)=z\exp\int_{0}^{z}\frac{\wp(\omega(t))-1}{t}dt,
	\end{equation*}
	where $w(z)={z(\sqrt{2/5}-z)}/{(1-\sqrt{2/5}z)}.$ \qed
\end{proof}
We now enlist below in the remark, certain special cases of earlier known results pertaining to our class $\mathscr{S}^*_{\wp}$:
\begin{remark}\label{2hankel}We obtain the following result by using a result (\cite{H22}, theorem 2.2, pg 230):
	\textit{	Let $f(z)=z+ \sum_{k=2}^{\infty}b_kz^k \in \mathscr{S}^*_{\wp}$, then}
	\begin{equation*}
	|H_2(2)|=|b_2b_4-{b_3}^2|\leq 1/4,
	\end{equation*}
	where equality is attained for the function $f_2$ given by \eqref{extremals}.
\end{remark}

\begin{remark}\label{3hankel} Now using Theorems \ref{coefficients}, \ref{bound} and Remark \ref{2hankel} together with the estimate given in \eqref{feketo} and triangle inequality, we obtain the following result:\\
	\textit{	Let the function $f(z)=z+ \sum_{k=2}^{\infty}b_kz^k \in \mathscr{S}^*_{\wp}$, then }
	$$|H_{3}(1)|\leq0.913864\cdots.$$
\end{remark}
\begin{remark}
	Until now,  the bound on third Hankel determinant is obtained using triangle inequality approach, but note that using the method applied in theorem \ref{13/18}, we can substantially improve the known bounds for many subclasses of starlike functions such as $\mathcal{S}^*_{s}$, $\mathcal{S}^*_{C}$ and $\mathcal{S}^*_{e}$.
\end{remark}	

\indent We know that $|H_{3}(1)|\leq1$  \cite{Zaprawa} for  $\mathcal{S}^*$, the class of starlike functions. Since, $\mathscr{S}^*_\wp\subset \mathcal{S}^*$,  it seems reasonable that the bound on $|H_3(1)|$ for $\mathscr{S}^*_\wp$ can be further improved. A function $f$ in $\mathcal{A}$ is called n-fold symmetric if $f(e^{2\pi i/n}z)=e^{2\pi i/n}f(z)$ holds for all $z\in \mathbb{D}$, where $n$ is a natural number. We denote the set of n-fold symmetric functions by $\mathcal{A}^{(n)}$. Let $f\in {\mathcal{A}}^{(n)}$, then $f$ has power series expansion $f(z)=z+b_{n+1}z^{n+1}+b_{2n+2}z^{2n+2}+\cdots$. Therefore, for $f\in {\mathcal{A}}^{(3)}$ and $f\in {\mathcal{A}}^{(2)}$, respectively, we have
\begin{equation}\label{fold40}
H_3(1)=-{b_4}^2 \quad\text{and}\quad H_3(1)=b_3(b_5-{b_3}^2).
\end{equation}
Thus we can now find estimates on the third Hankel determinant $|H_3(1)|$ in the classes $\mathscr{S}^{*(2)}_{\wp}$ and $\mathscr{S}^{*(3)}_{\wp}$.

\begin{theorem}\label{n-fold}
	Let $f\in\mathscr{S}^*_\wp$. Then
	\begin{itemize}
		\item [$(i)$] $\hat{f}\in\mathscr{S}^{*(3)}_{\wp}$ implies that $|H_3(1)|\leq1/9$.
		\item [$(ii)$] $\hat{f}\in\mathscr{S}^{*(2)}_{\wp}$ implies that $|H_3(1)|\leq1/16$.
	\end{itemize}
	The result is sharp.	
\end{theorem}
\begin{proof}
	(i) Since $f(z)=z+b_2z^2+\cdots \in \mathscr{S}^{*}_{\wp}$ if and only if $\hat{f}(z)= (f(z^3))^{1/3}=z+\beta_4z^4+\cdots \in \mathscr{S}^{*(3)}_{\wp}$. We have $\beta_4=b_2/3$. Hence for $\hat{f}\in\mathscr{S}^{*(3)}_{\wp}$, from \eqref{coef} and \eqref{fold40},  we obtain
	\begin{equation*}
	|H_3(1)|=|\beta_4|^2=\frac{1}{9}|b_2|^2\leq\frac{1}{9}.
	\end{equation*}
	The above estimate is sharp for $\hat{f}_1$, where $f_1$ is given by \eqref{extremals}.\\
	(ii) Since $f(z)=z+b_2z^2+\cdots\in \mathscr{S}^*_\wp$ if and only if $\hat{f}(z)=(f(z^2))^{1/2}=z+\alpha_3z^3+\alpha_5z^5+\cdots \in \mathscr{S}^{*(2)}_{\wp}$.
	Upon comparing the coefficients in the following:
	\begin{equation*}
	z^2+b_2z^4+b_3z^6+\cdots=(z+\alpha_3z^3+\alpha_5z^5+\cdots)^2,
	\end{equation*}
	we obtain
	\begin{equation}\label{s2}
	\alpha_3=\frac{1}{2}b_2\quad\text{and}\quad \alpha_5=\frac{1}{2}b_3-\frac{1}{8}{b_2}^2.
	\end{equation} 
	If $\hat{f}\in\mathscr{S}^{*(2)}_{\wp}$, then from \eqref{fold40}, we have
	\begin{equation*}
	H_3(1)=\alpha_3(\alpha_5-{\alpha_3}^2).
	\end{equation*}
	Now using \eqref{a15}, \eqref{s2} and  Lemma~\ref{2}, we obtain
	\begin{align*}
	|H_3(1)| =\frac{1}{4}\left|b_2\left(b_3-\frac{3}{4}{b_2}^2\right)\right|
	=\frac{1}{64}|p_1||({p_1}^2-p_1)+\xi(4-{p_1}^2)|,
	\end{align*}
	where $|\xi|\leq1$. Since $H_3(1)=\alpha_3(\alpha_5-{\alpha_3}^2)$ is rotationally invariant, so we may assume $p_1:=p\in [0,2]$. Thus using triangle inequality, we easily get $|H_3(1)|\leq (3p^3-4p^2+4p)/256=:g(p).$ Since $g'(p)>0$ for all $p\in[0,2]$. Therefore, $\max_{0\leq p\leq2}g(p)=g(2)$. Hence $$|H_3(1)|\leq\frac{1}{16}.$$
	The above estimate is sharp for $\hat{f}_1$, where $f_1$ is given by \eqref{extremals}. \qed
\end{proof}
In the following result, the bound obtained in the Remark \ref{3hankel} is improved.
\begin{theorem}\label{13/18}
	Let $f\in\mathscr{S}^*_\wp$. Then $|H_3(1)|\leq0.150627$.
\end{theorem}
\begin{proof}
	From  \eqref{123} and \eqref{a15}, we have
	\begin{align*}\label{h1}
	H_3(1)&=\frac{1}{9216}(-21{p}^6_1+60{p_1}^4p_2+96{p_1}^3p_3+192p_1p_2p_3\\
	&\quad-144{p_1}^2{p_2}^2-144{p_1}^2p_4-72{p_2}^3-256{p_3}^2+288p_2p_4)
	\end{align*}
	and using Lemma \ref{2} and writing $p_1$ as $p$ and $t=4-{p}^2_1$, we have
	\begin{equation}\label{h2}
	H_3(1)=\frac{1}{9216}\left({\Upsilon_1}(p,\zeta)+{\Upsilon_2}(p,\zeta)\eta+{\Upsilon_3}(p,\zeta){\eta}^2+{\Upsilon_4}(p,\zeta,\eta)\xi\right),
	\end{equation}
	where $\zeta, \eta, \xi \in \overline{\mathbb{D}}$ and
	\begin{align*}
	\Upsilon_1(p,\zeta) &=-4p^6+t(t(-25p^2{\zeta}^2+19p^2{\zeta}^3+2p^2{\zeta}^4+36{\zeta}^3)+5p^4\zeta\\
	&\quad-16p^4{\zeta}^2-24p^2{\zeta}^3),\\
	\Upsilon_2(p,\zeta)&=t(1-|\zeta|^2)(t(64p{\zeta}^2-80p\zeta)+32p^3), \\
	\Upsilon_3(p,\zeta) &=-t^2(1-|\zeta|^2)(64+8|\zeta|^2)\\
	\Upsilon_4(p,\zeta,\eta)&= 72t^2(1-|\zeta|^2)^2\zeta.
	\end{align*}
	Let $x=|\zeta|\in[0,1]$ and $y=|\eta|\in[0,1]$. Now using $|\xi|\leq1$ and triangle inequality, from (\ref{h2}) we obtain
	\begin{align}\label{h3}
	|H_3(1)|&\leq \frac{1}{9216}\left(f_1(p,x)+f_2(p,x)y+f_3(p,x)y^2+f_4(p,x)\right)\\
	&=:\frac{F(p,x,y)}{9216},
	\end{align}
	where
	\begin{align*}
	f_1(p,x)&=4p^6+t(t(25p^2x^2+19p^2x^3+2p^2x^4+36x^3)\\
	&\quad+5p^4x+16p^4x^2+24p^2x^3),\\
	f_2(p,x)&=t(1-x^2)(t(80px+64px^2)+32p^3),\\
	f_3(p,x)&=t^2(1-x^2)(64+8x^2)\\
	\text{and}\quad f_4(p,x)&=72t^2x(1-x^2)^2.
	\end{align*}
	Since $f_2(p,x)$ and $f_3(p,x)$ are non-negative functions over $[0,2]\times[0,1]$. Therefore, from (\ref{h3}) together with $y=|\eta|\in[0,1]$, we obtain
	\begin{equation*}
	F(p,x,y)\leq F(p,x,1).
	\end{equation*}
	Thus, $F(p,x,1)=f_1(p,x)+f_2(p,x)+f_3(p,x)+f_4(p,x)=: G(p,x).$\\
	Now we shall maximize $G(p,x)$ over $[0,2]\times[0,1]$. For this we consider the following possible cases:
	\begin{enumerate}[(i)]
		\item when $x=0$, we have
		\begin{equation*}
		G(p,0)=1024-512p^2+128p^3+64p^4-32p^5+4p^6=:g_1(p).
		\end{equation*}
		Since $g'_1(p)<0$ on $[0,2]$. Therefore, $g_1(p)$ is an decreasing function over $[0,2]$. Thus, the function $g_1(p)$ attains its maximum value at $p=0$ which is equal to $1024$.
		\item when $x=1$, we have
		\begin{equation*}
		G(p,1)=576+544p^2-272p^4+29p^6=:g_2(p).
		\end{equation*}
		Since $g'_2(p)=0$ has a critical point at $p_0=2\sqrt{(68-7\sqrt{34})/87}\approx1.11795$. Therefore, it is easy to see that $g_2(p)$ is an increasing function for $p\leq p_0$ and decreasing for $p_0\leq p$. Thus, the function $g_2(p)$ attains its maximum at $p:=p_0,$ which is approximately equal to $887.674$.
		\item when $p=0$, we have
		\begin{equation*}
		G(0,x)=1024-896x^2+576x^3-128x^4=:g_3(x).
		\end{equation*}
		Since $g'_3(x)<0$ on $[0,1]$. Therefore, the function $g_3(x)$ attains its maximum at $x=0$, which is equal to $1024$ and for the case, when $p=2$, we easily obtain $G(p,x)\leq256$.
		\item when $(p,x)\in(0,2)\times(0,1)$, a numerical computation shows that there exists a unique real solution for the system of equations
		$$\partial{G(p,x)}/\partial{x}=0\quad \text{and}\quad \partial{G(p,x)}/\partial{p}=0$$
		inside the rectangular region: $[0,2]\times[0,1]$,   at $(p,x)\approx(0.531621, 0.482768)$. Consequently, we obtain $G(p,x)\leq1388.18$.
	\end{enumerate}
	Hence, from the above cases we conclude that
	\begin{equation*}
	F(p,x,y)\leq1388.18 \quad\text{on}\quad[0,2]\times[0,1]\times[0,1],
	\end{equation*}
	which implies that
	\begin{equation*}
	H_3(1)\leq\frac{1}{9216}F(p,x,y)\leq0.150627.
	\end{equation*}
	Hence the result. \qed
\end{proof}

\noindent \textbf{Conjecture.}
If $f\in\mathscr{S}^*_{\wp}$, then the sharp bound for the third Hankel determinant is given by
\begin{equation*}
|H_3(1)|\leq \frac{1}{9}\approx0.1111\cdots,
\end{equation*}
with the extremal function $f(z)=z\exp\left(\frac{1}{3}(e^{z^3}-1)\right)=z+\frac{1}{3}z^4+\frac{2}{9}z^7+\cdots.$

\section*{Conflict of interest}
	The authors declare that they have no conflict of interest.

\end{document}